\documentclass{article}

\usepackage{amsmath}
\usepackage{amssymb}
\usepackage{amsthm}
\usepackage{latexsym}
\usepackage{bbold}

\usepackage[all]{xy}
\usepackage{amssymb,amsfonts}
\usepackage{pb-diagram}
\usepackage{amsmath,amsthm,amsfonts,amssymb}
\usepackage{graphicx}
\usepackage{extpfeil} 
\usepackage{dsfont} 

\usepackage{soul}
\usepackage{hyperref}
\usepackage[usenames]{color}
\usepackage{colortbl}

\newtheorem{theorem}{Theorem}[section]
\newtheorem{proposition}[theorem]{Proposition}
\newtheorem{corollary}[theorem]{Corollary}
\newtheorem{lemma}[theorem]{Lemma}

\theoremstyle{definition} 
\newtheorem{definition}[theorem]{Definition}

\newtheorem{remark}[theorem]{Remark}
\newtheorem{notation}[theorem]{Notation}
\newtheorem*{problem}{Problem}

\newenvironment{romanenumerate}
	{\begin{enumerate}

	}
        {
        
        \end{enumerate}
        } 
\newenvironment{Romanenumerate}
	{\begin{enumerate}

	}
        {
        
        \end{enumerate}
        } 
\newenvironment{arabicenumerate}
	{\begin{enumerate}

	}
        {
        
        \end{enumerate}
        } 

\newcommand{\sst}{\; | \;}  
\newcommand{\gst}{\, | \,}  
\newcommand{\integers}{\ensuremath{\mathds{Z}}} 
\newcommand{\naturals}{\ensuremath{\mathds{N}}} 
\newcommand{\cs}{,\;} 
\newcommand{\qs}{\:} 
\newcommand{\nclofin}[2]{\mathrm{ncl}_{{#2}}\left({#1}\right)} 
\newcommand{\GammaPresentation}{\langle A\,|\,\mathcal{R}\rangle} 
\newcommand{\GPresentation}{\ensuremath{\langle Z\gst S\rangle}}
\newcommand{\Hom}{\mathrm{Hom}}
\newcommand{\wl}[1]{|{#1}|} 

\newcommand{\braced}[4]{\left\{\begin{array}{ll} {#1} & {#2} \\ {#3} & {#4} \end{array}\right.}
\newcommand{\bracedTwo}{\braced}

\newcommand{\comment}[1]{} 

\title{Effective embedding of residually hyperbolic groups into direct products of extensions of centralizers}
\author{Olga Kharlampovich\footnote{Hunter College, CUNY and McGill University, supported by CUNY grant;
e-mail: okharlampovich@gmail.com}\phantom{i} and Jeremy Macdonald\footnote{Carleton University; e-mail: jeremodm@gmail.com}\\
}
\begin{document}

\maketitle{}

\begin{abstract}
For any torsion-free hyperbolic group $\Gamma$ and any group $G$ that is fully residually $\Gamma$, we construct algorithmically
a finite collection of homomorphisms
from $G$ to groups obtained from $\Gamma$ by extensions of centralizers, at least one of which is injective.  When $G$ is residually $\Gamma$, this gives a effective embedding
of $G$ into a direct product of such groups.  We also give an algorithmic construction of a diagram encoding the set of homomorphisms from a given finitely presented group
to $\Gamma$.
\end{abstract}

\section{Introduction}
A group $G$ is \emph{discriminated} by another group $\Gamma$ (or is \emph{fully residually $\Gamma$}) if for every finite set of non-trivial elements $\{g_{1},\ldots, g_{n}\}$ of
$G$ there exists a homomorphism $\phi: G\rightarrow \Gamma$ such that $\phi(g_{i})\neq 1$ for $i=1,\ldots, n$.  If this condition is only required to hold for $n=1$ we say that $G$ is \emph{separated} by $\Gamma$
(or is \emph{residually $\Gamma$}).

The class of fully residually \emph{free} groups (when $\Gamma$ is a non-abelian free group) has been extensively studied in the last 15 years, particularly in connection with
\emph{Tarski's problems} on the elementary theory of a free group (\cite{KM06}, \cite{Sel06}), and now have a well-developed theory.  Many algorithmic problems related to
these groups have been solved in recent years.

Generalizing to the case when $\Gamma$ is a hyperbolic group, much of the theory has been developed and is similar, but many algorithmic questions remain open. This paper is motivated by the following problem.

\begin{problem} Is the elementary theory $\mathrm{Th}(\Gamma)$ of a (torsion-free) hyperbolic group $\Gamma$  decidable? \end{problem}

Notice, that it was proved in \cite{KM05JSJ} that the universal theory of a finitely generated
fully residually free group is decidable and in \cite{Dah09}
that the universal theory of a  hyperbolic group is
decidable. We will give another proof of this result (for torsion-free hyperbolic groups) in Corollary~\ref{Cor:UnivTh}.

Fix throughout this paper a torsion-free hyperbolic group $\Gamma = \GammaPresentation$.
One of the characterizations of finitely generated groups $G$ discriminated by $\Gamma$ is that they embed
into the \emph{Lyndon completion} $\Gamma^{\integers[t]}$ of $\Gamma$, or equivalently, into a group obtained from $\Gamma$ by a series of extensions of centralizers \cite{KM09}.
If $G$ is only separated by $\Gamma$, there is an embedding into a finite direct product of such groups.  The case when $\Gamma$ is a free group was proved by
Kharlampovich and Myasnikov, who also provided an algorithm to construct the embedding (in both the discriminated and separated cases) \cite{KM98b}.

For the general (torsion-free hyperbolic) case, however, the embedding described in \cite{KM09} is not effective.  We provide an algorithm to
construct the embedding of any residually $\Gamma$ group $G$ into a direct product of groups obtained from $\Gamma$ by extensions of centralizers.  When $G$ is
fully residually $\Gamma$, we effectively construct a finite collection of homomorphisms from $G$ into groups obtained from $\Gamma$ by extensions of centralizers, at least
one of which is an embedding (Theorem~\ref{Thm:FinitelyManyEmbeddings}).

The first step of our approach is to use \emph{canonical representatives} for certain elements of $\Gamma$, developed by Rips and Sela in their study of equations
over hyperbolic groups \cite{RS95}, to reduce part of the problem to the free group case.  As a corollary of this reduction, we are able to effectively construct a finite diagram
that describes the complete set $\mathrm{Hom}(G, \Gamma)$ of homomorphisms from an arbitrary finitely presented group $G$ to $\Gamma$
(Theorem~\ref{Thm:EffectiveSolutions}).

\subsection{Algebraic geometry over groups}

We will use throughout the language of \emph{algebraic geometry over groups} \cite{BMR99}.  We recall here some important notions and establish notation.

Let $\Gamma$ be a group generated by a finite set $A$ (`constants'), let $X$ be a finite set (`variables'), and set $\Gamma[X]=\Gamma\ast F(X)$.  Let
$\mathrm{Hom}_{\Gamma}(\Gamma[X], \Gamma)$ denote the set of homomorphisms from $\Gamma[X]$ to $\Gamma$ that are identical on $\Gamma$
(`$\Gamma$-homomorphisms').

To each element $s$ of $\Gamma[X]$ we associate a formal expression `$s=1$' called an
\emph{equation over $\Gamma$}.  A \emph{solution} to `$s=1$' is a homomorphism $\phi\in \mathrm{Hom}_{\Gamma}(\Gamma[X], \Gamma)$ such that
$s^{\phi}=1$.  A subset $S$ of $\Gamma[X]$ corresponds to the \emph{system of equations} `$S=1$' which we also denote `$S(X,A)=1$'.  Define
\[
\Gamma_{S} = \Gamma[X] / \nclofin{S}{\Gamma[X]},
\]
where $\nclofin{S}{\Gamma[X]}$ is the normal closure of $S$ in $\Gamma[X]$,
and note that every solution to the system $S$ factors through $\Gamma_{S}$.  If $\Gamma$ has the presentation $\GammaPresentation$, then
\[
\Gamma_{S} \simeq \langle X,A \, | \, S, \mathcal{R}\rangle.
\]
Define the \emph{radical} $R_{\Gamma}(S)$ of $S$ over $\Gamma$ by
\[
R_{\Gamma}(S) = \{ t\in \Gamma[X] \; | \; \forall_{\phi\in \mathrm{Hom}_{\Gamma}(\Gamma[X], \Gamma)}\, \forall_{s\in S} \; (s^{\phi}=1 \implies t^{\phi}=1) \}
\]
and define the \emph{coordinate group} of $S$ over $\Gamma$ by
\[
\Gamma_{R_{\Gamma}(S)} = \Gamma[X] / R_{\Gamma}(S).
\]
Every solution to $S$ factors through $\Gamma_{R_{\Gamma}(S)}$.  We frequently omit the subscript $\Gamma$ and write $\Gamma_{R(S)}$ instead
of $\Gamma_{R_{\Gamma}(S)}$.

In many cases, we will encounter systems of equations $S$ where $\Gamma_{S}$ is (fully) residually $\Gamma$.  In this case, the radical and normal closure coincide.

\begin{lemma}\label{Lem:ResidualNullstellensatz}
Let $S(X)$ be a system of equations over a group $G$.  If $G_{S}$ is residually $G$, then
\[
R_{G}(S) = \nclofin{S}{G[X]}
\]
and hence $G_{R(S)} = G_{S}$.
\end{lemma}
\begin{proof}
It is always the case that $\nclofin{S}{G[X]}\subset R_{G}(S)$, so assume for contradiction that there exists $w\in R_{G}(S)\setminus \nclofin{S}{G[X]}$.
Then $w\neq 1$ in $G_{S}$, so there exists a homomorphism $\phi:G_{S}\rightarrow G$ such that $w^{\phi}\neq 1$. But $\phi$ is a solution to $S$
and $w\in R_{G}(S)$ so $w^{\phi} =1$, a contradiction.
\end{proof}


When $S$ is a subset of $F(X)$ we say that the system is \emph{coefficient-free} and we may consider $F(X)$ in place of $\Gamma\ast F(X)$ and
the ordinary set of homomorphisms $\mathrm{Hom}(F(X), \Gamma)$ in place of $\mathrm{Hom}_{\Gamma}(\Gamma[X], \Gamma)$.  In particular, for any
group $G$ presented by $\GPresentation$ we may consider $S$ as a system of equations in variables $Z$.  In the general case, when $S\subset \Gamma[X]$, we may
consider $S$ as a system of equations over any group $G$ that has $\Gamma$ as a fixed subgroup (i.e.  any $G$ in the `category of $\Gamma$-groups').

We will study the situation in which the same system of equations is considered over both a free group $F=F(A)$ and one of its quotients
$\Gamma=\langle A\,|\, \mathcal{R}\rangle$.
Fix $\pi:F\rightarrow \Gamma$ the
canonical epimorphism.
The map $\pi$ induces an epimorphism $F[X]\rightarrow \Gamma[X]$, also denoted $\pi$, by fixing each $x\in X$.  For a system of equations
$S\subset F[X]$, we study the corresponding system $S^{\pi}\subset\Gamma[X]$.

\subsection{Toral relatively hyperbolic groups}
A group $G$ that is hyperbolic relative to a collection $\{H_{1},\ldots,H_{k}\}$ of subgroups is called \emph{toral} if $H_{1},\ldots,H_{k}$ are all finitely-generated abelian groups and $G$ is
torsion-free.
Many algorithmic problems in
(toral) relatively hyperbolic groups are decidable, and in particular we take note of the following for later use.
\begin{lemma}\label{Lem:AlgorithmsRelativelyHyperbolic}
In every toral relatively hyperbolic group $G$, the following hold.
\begin{arabicenumerate}
\item The conjugacy problem in $G$, and hence the word problem, is decidable.
\item $G$ satisfies the CSA property (maximal abelian subgroups are malnormal), hence commutation is a transitive relation on the set of non-trivial elements of $G$.
\item There exists an algorithm that, given an element $g\in G$, produces a generating set for $C(g)$.
\end{arabicenumerate}
\end{lemma}
\begin{proof}
The word problem was solved in \cite{Far98} and the conjugacy problem in \cite{Bum04}.  The CSA property is proved in \cite{KM09} and in \cite{Gro09}, and
commutation-transitivity follows from CSA \cite{MR96}.

For the third statement, we first determine, using Theorem~5.6 of \cite{Osi06Memoirs}, whether $g$ is parabolic (conjugate to an element of some $H_{i}$)
or hyperbolic (not conjugate to any element of any $H_{i}$).  If $g$ is parabolic, this theorem also produces $h\in G$ and an index $i$ such that $g\in H_{i}^{h}$.
Since $G$ is commutation-transitive, it follows that $C(g)$ is a maximal abelian subgroup hence is equal to $H_{i}^{g}$.  Then the conjugates of
the generators of $H_{i}$ form a generating set for $C(g)$.

If $g$ is hyperbolic, then $C(g)$ is infinite cyclic.  Indeed, Theorem~4.3 of \cite{Osi06IJAC} shows that the subgroup
\[
E(g)= \{ h\in G \sst \exists\qs n\in\naturals :\qs h^{-1}g^{n}h=g^{n}\}
\]
has a cyclic subgroup of finite index.  Since $G$ is torsion-free, $E(g)$ must be infinite cyclic (see for example the proof of
Proposition~12 of \cite{MR96}).  Clearly $C(g)\leq E(g)$, hence $C(g)$ is infinite cyclic.

To construct a generator for $C(g)$,
we use the fact that the \emph{root problem} for hyperbolic elements is decidable in $G$: there is an algorithm that,
given $g\in G$ decides whether or not there exist $f\in G$ and $n>1$ such that $g=f^{n}$ (Theorem~5.17 of \cite{Osi06Memoirs}).

If no such $f$ and $n$ exist, then $g$ generates $C(g)$.  Otherwise, we may obtain $f$ and $n$ by exhaustive search.  We now
repeat this procedure with $f$ in place of $g$.  The process must terminate, yielding a generator of $C(g)$, since $C(g)\simeq \integers$.
\end{proof}


\section{Effective description of homomorphisms to $\Gamma$}

In this section, we describe an algorithm that takes as input a system
of equations $S$ over $\Gamma$ and produces a tree diagram $\mathcal{T}$ that encodes the set
$\Hom_{\Gamma}(\Gamma_{S},\Gamma)$.  When $S$ is a system without coefficients, we interpret $S$ as relators for a finitely presented group
$G=\GPresentation$ and the diagram $\mathcal{T}$ encodes instead the set $\Hom(G,\Gamma)$.

Though the diagram $\mathcal{T}$ will give a finite description of $\Hom_{\Gamma}(\Gamma_{S},\Gamma)$, it is not a `Makanin-Razborov diagram' in the
sense of \cite{Gro05}.  We discuss this further at the end of this section.

There are two ingredients in this construction: first,
the reduction of the system $S$ over $\Gamma$ to finitely many systems of equations over free groups, and second, the construction of Hom-diagrams (Makanin-Razborov diagrams)
for systems of equations over free groups.

\begin{notation}\label{Not:Lifting}
Let $\overline{\phantom{c}}$ denote the canonical epimorphism $F(Z,A)\rightarrow \Gamma_{S}$.
For a homomorphism $\phi: F(Z,A)\rightarrow K$ we define $\overline{\phi}: \Gamma_{S} \rightarrow K$
by
\[
\big(\overline{w}\big)^{\overline{\phi}} =  w^{\phi},
\]
where any preimage $w$ of $\overline{w}$ may be used.  We will always ensure that $\overline{\phi}$ is a well-defined homomorphism.
For a system $S(Z)=1$ without coefficients, $\overline{\phantom{c}}$ denotes the canonical epimorphism $F(Z)\rightarrow \langle Z\gst S\rangle$
and $\overline{\phi}$ is defined analogously.
\end{notation}


\subsection{Reduction to systems of equations over free groups}
In \cite{RS95}, the problem of deciding whether or not a system of equations $S$ over a torsion-free hyperbolic group $\Gamma$ has a solution was solved by constructing
\emph{canonical representatives} for certain elements of $\Gamma$. This construction reduced the problem to deciding the existence of solutions in finitely many
systems of equations over free groups, which had been previously solved.  The reduction may also be used to find all solutions to $S$ over $\Gamma$, as described
below.

\begin{lemma}\label{Lem:RipsSela1}
Let $\Gamma=\GammaPresentation$ be a torsion-free $\delta$-hyperbolic group and $\pi : F(A)\rightarrow \Gamma$ the canonical epimorphism.  There
is an algorithm that, given a system $S(Z,A)=1$
of equations over $\Gamma$, produces finitely many systems of
equations
\begin{equation}
S_{1} (X_{1},A)=1,\ldots,S_{n}(X_{n},A)=1
\end{equation}
over $F$ and homomorphisms $\rho_{i}: F(Z,A)\rightarrow F_{S_{i}}$ for $i=1,\ldots,n$
such that
\begin{romanenumerate}
\item for every $F$-homomorphism $\phi : F_{S_{i}}\rightarrow F$,  the map $\overline{\rho_{i}\phi\pi}:\Gamma_{S}\rightarrow \Gamma$ is a $\Gamma$-homomorphism, and
\item for every $\Gamma$-homomorphism $\psi: \Gamma_{S}\rightarrow \Gamma$ there is an integer $i$ and an $F$-homomorphism
$\phi : F_{S_{i}}\rightarrow F(A)$ such that $\overline{\rho_{i}\phi\pi}=\psi$.
\end{romanenumerate}
Further, if $S(Z)=1$ is a system without coefficients, the above holds with $G=\GPresentation$ in place of $\Gamma_{S}$ and `homomorphism' in place of
`$\Gamma$-homomorphism'.
\end{lemma}

\begin{proof}
The result is an easy corollary of Theorem~4.5 of \cite{RS95}, but we will provide a few details.
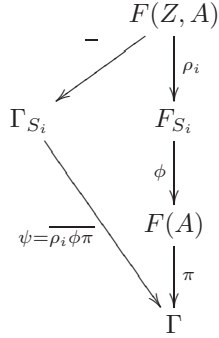
\begin{figure}[htbp]
\begin{center}
\[
\xymatrix{
&  F(Z,A) \ar[ld]_{\overline{\phantom{\phi}}} \ar[d]^{\rho_{i}} \\
\Gamma_{S_{i}}   \ar[ddr]_{\psi=\overline{\rho_{i}\phi\pi}} & F_{S_{i}}  \ar[d]_{\phi}  \\
& F(A) \ar[d]^{\pi} \\
&   \Gamma
}
\]
\caption{Commutative diagram for Lemma~\ref{Lem:RipsSela1}.}
\label{Figure:CanonicalReps}
\end{center}
\end{figure}

The main tool used in \cite{RS95} is the construction of certain representatives in $F$ for elements of $\Gamma$ (Definition~3.10).
For every $m\in\naturals$ and $g\in \Gamma$ a word $\theta_{m}(g)\in F$ is constructed such that
\[
\theta_{m}(g)=g \mbox{ in } \Gamma.
\]
The construction depends only on $m$ and $g$ and may be done algorithmically.

For most $g$ and values of the parameter $m$, the representatives enjoy no special properties: rather, for sets of words $\{g_{i}\}$ satisfying
a system of triangular equations, they satisfy certain cancellation properties (see below) for at least one $m_{0}$
in a bounded interval.  In this case, the representatives $\theta_{m_{0}}(g_{i})$ are often referred to as \emph{canonical representatives} (though this term
is used in \cite{RS95} to refer to all the words $\theta_{m}(g)$).

We may assume that the system $S(Z,A)$, in variables $z_{1},\ldots,z_{l}$, consists of $m$ constant equations and $q-m$ triangular equations, i.e.
\[
S(Z,A) = \braced{z_{\sigma(j,1)}z_{\sigma(j,2)}z_{\sigma(j,3)}=1}{j=1,\ldots,q-m}{z_{s}  =  a_{s}}{s=l-m+1,\ldots,l}
\]
where $\sigma(j,k)\in\{1,\ldots,l\}$ and $a_{i}\in\Gamma$.

Let\footnote{The constant of hyperbolicity $\delta$ need not be provided in the input, as it may be computed from a presentation of $\Gamma$ using the results of \cite{EH01}.} $L=q\cdot 2^{5050(\delta+1)^{6}(2|A|)^{2\delta}}$.
Suppose $\psi: F(Z,A)\rightarrow \Gamma$ is a solution of $S(Z,A)$ and denote
\[
\psi(z_{\sigma(j,k)})=g_{\sigma(j,k)}.
\]
Then there exist
$h_{k}^{(j)}, c_{k}^{(j)}\in F(A)$ (for $j=1,\ldots,q-m$ and $k=1,2,3$) such that
\begin{romanenumerate}
\item each $c_{k}^{(j)}$ has length less than\footnote{The bound of $L$ here, and below, is extremely loose.  Somewhat tighter, and more intuitive, bounds are given in \cite{RS95}.} $L$
 (as a word in $F$), \label{RepsCond1}
\item $c_{1}^{(j)}c_{2}^{(j)}c_{3}^{(j)}  =  1$ in $\Gamma$, \label{RepsCond2}
\item there exists $m\leq L$\ such that the canonical representatives satisfy the following equations in $F$:\label{RepsCond3}
\begin{eqnarray}
\theta_{m} (g_{\sigma(j,1)}) & = & h_{1}^{(j)} c_{1}^{(j)} \left(h_{2}^{(j)}\right)^{-1} \label{CanonReps1}\\
\theta_{m} (g_{\sigma(j,2)}) & = & h_{2}^{(j)} c_{2}^{(j)} \left(h_{3}^{(j)}\right)^{-1}\\
\theta_{m} (g_{\sigma(j,3)}) & = & h_{3}^{(j)} c_{3}^{(j)} \left(h_{1}^{(j)}\right)^{-1}.\label{CanonReps3}
\end{eqnarray}
\end{romanenumerate}
In particular, when $\sigma(j,k)=\sigma(j',k')$ (which corresponds to two occurrences in $S$ of the variable $z_{\sigma(j,k)}$) we have
\begin{equation}
h_{k}^{(j)} c_{k}^{(j)} \left(h_{k+1}^{(j)}\right)^{-1} = h_{k'}^{(j')} c_{k'}^{(j')} \left(h_{k'+1}^{(j')}\right)^{-1}.\label{Hequality}
\end{equation}

Consequently, we construct the systems $S(X_{i},A)$ as follows.
For every positive integer $m\leq L$ and every choice of $3(q-m)$ elements $c_{1}^{(j)},c_{2}^{(j)},c_{3}^{(j)}\in F$  ($j=1,\ldots,q-m$)
satisfying (i) and (ii)
we build a system $S(X_{i},A)$
consisting of the equations
\begin{eqnarray}
x_{k}^{(j)}c_{k}^{(j)}\left(x_{k+1}^{(j)}\right)^{-1} & = & x_{k'}^{(j')}c_{k'}^{(j')}\left(x_{k'+1}^{(j')}\right)^{-1} \label{Eqn:SC1}\\
x_{k}^{(j)}c_{k}^{(j)}\left(x_{k+1}^{(j)}\right)^{-1} & = & \theta_{m}(a_s) \label{Eqn:SC2}
\end{eqnarray}
where an equation of type   (\ref{Eqn:SC1}) is included whenever  $\sigma(j,k)=\sigma(j',k')$ and an equation of type (\ref{Eqn:SC2}) is included whenever
$\sigma(j,k)=s\in\{l-m+1,\ldots,l\}$.  To define $\rho_{i}$, set
\[
\rho_i (z_s) = \bracedTwo{x_{k}^{(j)}c_{k}^{(j)}\left(x_{k+1}^{(j)}\right)^{-1},}{1\leq s \leq l-m \mbox{ and } s=\sigma(j,k)}{ \theta_{m}(a_s),}{l-m+1\leq s \leq l}
\]
where for $1\leq s \leq l-m$ any $j,k$ with $\sigma(j,k)=s$ may be used.

If $\psi:F(Z)\rightarrow \Gamma$ is any solution to $S(Z,A)=1$, there is a system $S(X_{i},A)$ such that $\theta_{m}(g_{\sigma(j,k)})$ satisfy
(\ref{CanonReps1})-(\ref{CanonReps3}).  Then the required solution $\phi$ is given by
\[
\phi\big(x_{j}^{(k)}\big) = h_{j}^{(k)}.
\]
Indeed, (iii) implies that $\phi$ is a solution to $S(X_{i},A)=1$.  For $s=\sigma(j,k)\in\{1,\ldots,l-m\}$,
\[
z_{s}^{\rho_{i}\phi} = h_{k}^{(j)} c_{k}^{(j)} \left(h_{k+1}^{(j)}\right)^{-1} = \theta_{m}(g_{\sigma(j,k)})
\]
and similarly for $s\in\{l-m+1,\ldots,l\}$, hence $\psi= \rho_{i}\phi\pi$.

Conversely, for any solution $\phi\big(x_{j}^{(k)}\big)= h_{j}^{(k)}$ of $S(X_{i})=1$ one sees that by (\ref{Eqn:SC1}),
\[
z_{\sigma(j,1)}z_{\sigma(j,2)}z_{\sigma(j,3)} \xmapsto{\rho_{i}\phi} h_{1}^{(j)} c_{1}^{(j)}c_{2}^{(j)}c_{3}^{(j)} \big(h_{1}^{(j)}\big)^{-1}
\]
which maps to 1 under $\pi$ by (ii), hence $\rho_{i}\phi\pi$ induces a homomorphism.
\end{proof}


\subsection{Encoding solutions with the tree $\mathcal{T}$}\label{section:HomDiagrams}
An algorithm is described in \S 5.6 of  \cite{KM05Implicit} which constructs, for a given system of equations $S(X,A)$ over the free group $F$,
a diagram encoding the set of solutions of $S$.  The diagram consists of a directed finite rooted tree $T$ with
the following properties.  Let $G=F_{R(S)}$.
\begin{romanenumerate}
\item Each vertex $v$ of $T$ is labelled by a pair  $(G_{v},Q_{v})$ where $G_{v}$ is an $F$-quotient of $G$ and $Q_{v}$ a finitely generated subgroup of $\mathrm{Aut}_{F}(G_{v})$.
The root $v_0$ is labelled by $(G,1)$ and every leaf is labelled by $(F(Y)\ast F,1)$ where $Y$ is some finite set (called \emph{free variables}).
Each $G_{v}$, except possibly $G_{v_{0}}$, is fully residually $F$.
\item Every (directed) edge $v\rightarrow v'$ is labelled by a proper surjective $F$-homomorphism $\pi(v,v'):G_{v}\rightarrow G_{v'}$.
\item For every $\phi\in\mathrm{Hom}_{F}(G,F)$ there is a path $p=v_0 v_1 \ldots v_k$ where $v_k$ is a leaf labelled by
$(F(Y)\ast F,1)$, elements $\sigma_{i}\in Q_{v_{i}}$, and a $F$-homomorphism
$\phi_{0}: F(Y)\ast F\rightarrow F$ such that
\begin{equation}
\phi = \pi(v_0,v_1) \sigma_1 \pi(v_1,v_2) \sigma_2 \cdots \pi(v_{k-2},v_{k-1})\sigma_{k-1}\pi(v_{k-1},v_{k})\phi_{0}.
\end{equation}
\end{romanenumerate}
The algorithm gives for each $G_{v}$ a finite presentation $\langle A_{v}|\mathcal{R}_{v}\rangle$, and for each $Q_{v}$ a finite list of
generators in the form of functions $A_{v}\rightarrow (A_{v}\cup A_{v}^{-1})^{*}$.  Note that the choices for $\phi_{0}$ are
exactly parametrized by the set of functions from $Y$ to $F$.

Let $S(Z,A)=1$ be a system of equations over $\Gamma$.  We will construct a diagram $\mathcal{T}$ to encode the set of solutions of $S(Z,A)=1$, as follows.

Apply Lemma~\ref{Lem:RipsSela1} to construct the systems $S_{1}(X_{1},A),\ldots,S_{n}(X_{n},A)$.
Create a root vertex $v_{0}$ labelled by $F(Z,A)$.  For each of the systems $S_{i}(X_{i},A)$, let $T_{i}$
be the tree constructed above. Build an edge from $v_{0}$ to the root of $T_{i}$ labelled by the homomorphism $\rho_{i} \pi_{S_{i}}$, where $\pi_{S_{i}}:F(X_{i},A)\rightarrow F_{R(S_{i})}$
is the canonical projection.
For each leaf $v$ of $T_{i}$, labelled by $F(Y)\ast F$, build a new vertex $w$ labelled by $F(Y)\ast\Gamma$ and an edge $v\rightarrow w$ labelled by the homomorphism
$\pi_{Y}:F(Y)\ast F\rightarrow F(Y)\ast \Gamma$ which is induced from $\pi:F\rightarrow \Gamma$ by acting as the identity on $F(Y)$.

Define a \emph{branch} $b$ of $\mathcal{T}$ to be a path $b=v_{0} v_{1} \ldots v_{k}$ from the root $v_{0}$ to a leaf $v_{k}$.
Let $v_{1}$ be labelled by $F_{R(S_{i})}$ and $v_{k}$ by $F(Y)\ast \Gamma$.
We associate to $b$ the set $\Phi_{b}$ consisting of all homomorphisms $F(Z)\rightarrow \Gamma$ of the form
\begin{equation}
\rho_{i}\pi_{S_{i}}\pi(v_{1},v_{2}) \sigma_{2} \cdots \pi(v_{k-2},v_{k-1})\sigma_{k-1}\pi(v_{k-1},v_{k})\pi_{Y}\phi
\end{equation}
where $\sigma_{j}\in Q_{v_{j}}$ and $\phi\in \mathrm{Hom}_{\Gamma}(F(Y)\ast\Gamma,\Gamma)$.
Since $\mathrm{Hom}_{\Gamma}(F(Y)\ast\Gamma,\Gamma)$ is in bijective correspondence with the set of functions $\Gamma^{Y}$, all elements of $\Phi_{b}$
can be effectively constructed.  We have obtained the following theorem.

\begin{theorem} \label{Thm:EffectiveSolutions}
There is an algorithm that, given a system $S(Z,A)=1$ of equations over $\Gamma$, produces a diagram encoding its set of solutions.  Specifically,
\[
\Hom(\Gamma_{R(S)},\Gamma) = \{ \overline{\phi} \sst \phi\in\Phi_{b}\cs \mbox{$b$ is a branch of $\mathcal{T}$}\}
\]
where $\mathcal{T}$ is the diagram described above.  When the system is coefficient-free, then the diagram encodes $\Hom(G, \Gamma)$ where
$G=\GPresentation$.
\end{theorem}

Note that in the diagram $\mathcal{T}$, the groups $G_{v}$ appearing at vertices are not quotients of coordinate group $\Gamma_{R(S)}$ and
that to obtain a homomorphism from $\Gamma_{R(S)}$ to $\Gamma$ one must compose maps along a complete path ending at a leaf of $\mathcal{T}$.  In \cite{Gro05} it
is shown that for any toral relatively hyperbolic group there exist Hom-diagrams with the property that every group $G_{v}$ is a quotient of
$\Gamma_{R(S)}$ and that every edge map $\pi(v,v')$ is a
proper surjective homomorphism.  However, we are not aware of a algorithm for constructing these
diagrams.


\section{Embedding into extensions of centralizers}

The proof given in \cite{KM09} that any fully residually $\Gamma$ group $G$ embeds into extensions of centralizers of $\Gamma$ involves two steps: first, $G$ is shown to
embed into the coordinate groups of an \emph{NTQ system} (see \S \ref{subsection:NTQ}), and second, such groups are shown to embed into extensions of centralizers of $\Gamma$.
The first step of this construction relies on the following theorem (Theorem 1.1 of \cite{Gro05}):
there exists a finite collection $\{ L_{i}\}$ of proper quotients of $G$ such that any homomorphism from
$G$  to $\Gamma$ factors through one of the $L_{i}$ (up to a certain equivalence).  Algorithmic construction of the set $\{ L_{i}\}$ is not given, nor are we aware of
an algorithm for constructing it.

Our construction avoids this first step while making use of the second step.  The outline of our procedure is as follows.
We construct the diagram $\mathcal{T}$ from \S 2 so that the groups in the first layer are coordinate groups over $F$ of NTQ systems \cite{KM98b}.
Regarding each $F$-NTQ system as a system of equations over $\Gamma$, we
obtain a collection of groups through which every homomorphism from $G$ to $\Gamma$ factors.  Doing so introduces some degeneracy in the
NTQ structure, which prevents these groups from embedding into centralizer extensions of $\Gamma$.  We compensate by eliminating the degenerate parts to
obtain `embeddable' $\Gamma$-NTQ groups.  Since $G$ is discriminated by $\Gamma$, it
must embed in at least one of these $\Gamma$-NTQ groups.  Each $\Gamma$-NTQ group then embeds into an
extension of centralizers of $\Gamma$, using techniques and results from the second step of the \cite{KM09} proof mentioned above.


\subsection{Quadratic equations and NTQ systems}\label{subsection:NTQ}

\subsubsection*{Quadratic equations}

An equation $s\in G[X]$ over a group $G$ is said to be (strictly) \emph{quadratic} if every variable appearing in $s$ appears at most (exactly) twice, and a system
of equations $S(X)\subset G[X]$ is (strictly) quadratic if every variable that appears in $S$ appears at most (exactly) twice.  Here we count both $x$ and $x^{-1}$ as
an appearance of $x$.
Constructing NTQ systems involves considerable analysis of quadratic equations, and is aided by considering certain standard forms.

\begin{definition}
A \emph{standard quadratic equation} over a group $G$ is an equation of one of the following forms, where $c_{i}$ and $d$ are all nontrivial
elements of $G$:
\begin{eqnarray}
\prod_{i=1}^{n}[x_{i},y_{i}] & = & 1, \;\;\; n \geq 1; \label{Eqn:st1}\\
\prod_{i=1}^{n}[x_{i},y_{i}] \prod_{i=1}^{m}c_{i}^{ z_{i}} d & = & 1,\;\;\;
n,m\geq 0, n+m \geq 1 ; \label{Eqn:st2}\\
\prod_{i=1}^{n}x_{i}^2 & = & 1, \;\;\; n \geq 1; \label{Eqn:st3}\\
\prod_{i=1}^{n}x_{i}^2 \prod_{i=1}^{m}c_{i}^{ z_{i}} d & = & 1, \;\;\; n,m
\geq 0, n+m \geq 1.\label{Eqn:st4}
\end{eqnarray}
The left-hand sides of the above equations are called the \emph{standard quadratic words}.
\end{definition}
The following result allows us to assume that quadratic equations always appear in standard form.
\begin{lemma}
Let $s(X)\in G[X]$ be a strictly quadratic word over a group $G$. Then there is
a $G$-automorphism $\phi$ such that  $s^{\phi}$ is a
standard quadratic word over $G$.
\end{lemma}
\begin{proof}
Follows easily from \S I.7 of \cite{LS77}.
\end{proof}

To each quadratic equation we associate a punctured surface. To (\ref{Eqn:st1}) we associate the orientable surface of genus $n$ and zero punctures, to (\ref{Eqn:st2})
the orientable surface of genus $n$ with $m+1$ punctures, to (\ref{Eqn:st3}) the non-orientable surface of genus $n$, and to (\ref{Eqn:st4}) the non-orientable surface of
genus $n$ with $m+1$ punctures.   For a standard quadratic equation $S$, denote by $\chi(S)$ the
Euler characteristic of the corresponding surface.

The words $[x_{i}, y_{i}]$, $x_{i}^{2}$, $c_{i}^{z_{i}}$, and $d$ are called \emph{atoms}.
Each atom of the form $[x_{i},y_{i}]$ contributes $-2$ to the Euler characteristic of $S$ while
$x_{i}^{2}$ and $c_{i}^{z_{i}}$ (as well as $d$) each contribute $-1$.
The standard quadratic equations with $\chi(S)>-2$ will appear as exceptional cases in our study.  We record them here for convenience:
\begin{description}
\item [Genus 0 forms:] $c^{z}d$, $c_{1}^{z_{1}} c_{2}^{z_{2}}d$
\item [Orientable forms:] $[x,y]$, $[x,y]d$
\item [Non-orientable forms:] $x^{2}$, $x^{2}d$, $x^{2}c^{z}d$, $x^{2}y^{2}$, $x^{2}y^{2}d$, $x^{2}y^{2}z^{2}$
\end{description}


\subsubsection*{NTQ systems}

Let $G$ be a group generated by $A$ and let $S(X,A)$ be a system of equations.  The system $S(X,A)$ is called
\emph{triangular quasi-quadratic} (TQ) if it can be partitioned into subsystems
\begin{eqnarray*}
S_{1}(X_{1}, X_{2}, \ldots, X_{n},A) & = & 1, \\
S_{2}(X_{2}, \ldots, X_{n},A) &  = & 1,\\
& \ldots & \\
S_{n}(X_{n},A) & = & 1
\end{eqnarray*}
where $\{X_{1},X_{2},\ldots,X_{n}\}$ is a partition of $X$, and such that for each $i$ one of the following holds:
\begin{Romanenumerate}
\item $S_{i}$ is quadratic in variables $X_{i}$; \label{NTQ1}
\item $S_{i} = \{[x,y]=1\cs [x,u]=1\sst x\cs y\in X_{i}\cs u\in U \}$ where $U\subset F(X_{i+1},\ldots, X_{n},A)$; \label{NTQ2}
\item $S_i = \{[x,y]=1\sst x\cs y\in X_i \}$; \label{NTQ3}
\item $S_i$ is empty. \label{NTQ4}
\end{Romanenumerate}

For any quadratic system $S$ over $G$ one can, by eliminating linear variables,
find a strictly quadratic system $S'$ over $G$ such that every variable occurs in exactly
one equation and $G_{S}\simeq G_{S'}$.  Consequently,  we may assume that every
system $S_{i}$ of $S$ that has the form (\ref{NTQ1}) consists of a single quadratic equation in standard form.
The number $n$ is called the \emph{depth} of the system.

For each $i$, define the coordinate group
\[
G_{i}=G[X_{i},\ldots,X_{n}] / R_{G}(S_{i},\ldots,S_{n})
\]
and set $G_{n+1}=G$. We interpret $S_{i}$ as a subset of $G_{i+1}*F(X_{i})$, i.e. letters from $X_{i}$ are considered variables and letters from $X_{i+1}\cup\ldots\cup X_{n}\cup A$ are considered as constants from $G_{i}$.
One may check that
for every $i=1,\ldots,n$,
\begin{equation}
G_{i} \simeq G_{i+1}[X_{i}] / R_{G_{i+1}}(S_{i}).
\end{equation}
This isomorphism in fact holds for any system of equations that can be partitioned in triangular form.

The system is called \emph{non-degenerate triangular quasi-quadratic} (NTQ) if for every $i$ the system $S_{i}(X_{i}, \ldots, X_{n},A)$ has a solution in
$G_{i+1}$ and the set $U$ in (\ref{NTQ2}) generates a centralizer in $G_{i+1}$.

\begin{definition}
A group $H$ is called a \emph{$G$-NTQ group} if there is a NTQ system $S$ over $G$ such that $H\simeq G_{R(S)}$.
\end{definition}


NTQ groups over free groups played a central role in the solution to Tarski's problems by Kharlampovich-Miasnikov and Sela.  In Sela's work, they are called
\emph{$\omega$-residually free towers} \cite{Sel01}.  While the present work concerns $\Gamma$-NTQ groups, we will study them as quotients
of certain $F$-NTQ groups.

\subsubsection*{Properties of $\mathbf{F}$-NTQ and $\mathbf{\Gamma}$-NTQ groups}

All of the TQ systems we will study satisfy $R_{G_{i+1}}(S_{i}) = \nclofin{S_{i}}{G_{i+1}[X_{i}]}$, hence $G_{i}$ will be presented by
\begin{equation}\label{Eqn:Gi}
G_{i} = \langle G_{i+1}, X_{i} \gst S_{i}\rangle
\end{equation}
and $G_{R(S)}$ by
\[
G_{R(S)} = \langle G, X\gst S\rangle.
\]
In this case, $G_{i}$ admits a graph of groups decomposition of one of the following four types, according to the form of $S_{i}$:
\begin{Romanenumerate}
\item as a graph of groups with vertices $v_{1}$, $v_{2}$ where $G_{v_1}=G_{i+1}$ and $G_{v_{2}}$ is the subgroup
$Q=\langle x_{1}, y_{1},\ldots, x_{n}, y_{n}, c_{1}^{z_{1}}, \ldots, c_{m}^{z_{m}}, d\rangle$ or
$Q=\langle x_{1}, \ldots, x_{n}, c_{1}^{z_{1}}, \ldots, c_{m}^{z_{m}}, d\rangle$, in the orientable and non-orientable cases respectively, $c_1,\ldots ,c_m,d\in G_{i+1}$
($Q$ is called a \emph{QH}-subgroup, cf. \cite{KM06} \S 3.8);
\item as a graph of groups with vertices $v_{1}$, $v_{2}$ where $G_{v_1}=G_{i+1}$, $G_{v_{2}}$ is a free abelian group of rank $m$,
and the edge groups generate a maximal abelian subgroup of $G_{v_1}$ (`rank $m$ extension of centralizer');
\item as a free product with a finite rank free abelian group;
\item as a free product with a finitely generated free group.
\end{Romanenumerate}

This structure allows us to use the graph of groups decomposition of $G_{i}$ to derive
properties of NTQ groups inductively.  In particular, we have the following.

\begin{lemma}\label{Lem:PropertiesOfGammaNTQ}
Let $\Gamma=\GammaPresentation$ be a toral relatively hyperbolic group and
$G$ a $\Gamma$-NTQ group such that $R_{G_{i+1}}(S_{i}) = \nclofin{S_{i}}{G_{i+1}[X_{i}]}$ for all $i=1,\ldots,n$.
Then $G$ is toral relatively hyperbolic. 
\end{lemma}
\begin{proof}
We proceed by induction on the height of the NTQ system.  The base $\Gamma=G_{n}$ is toral relatively hyperbolic.
Now assume that $G_{2}$ is toral relatively hyperoblic.  We will show that $G_{1}$ is toral relatively hyperbolic by
applying Theorem~0.1 of \cite{Dah03} (`Combination theorem') to the four possible decompositions of
$G_{1}$ described above.

Cases (\ref{NTQ4}) and (\ref{NTQ3}) follow from Theorem~0.1 parts (3) and (2), respectively, by amalgamating over the trivial subgroup.  Note that to use Theorem~0.1 (2) we
need the fact that
if $G$ is hyperbolic relative to the collection of subgroups $\mathcal{H}$ then it is also hyperbolic relative to $\mathcal{H}\cup \{1\}$.
Case (\ref{NTQ2}) follows from Theorem~0.1 (2) by amalgamating over $P=\langle U_{i}\rangle$, which is maximal abelian in $G_{2}$.

For case (\ref{NTQ1}), consider first the case when the surface corresponding to the quadratic equation has punctures.  In this case we form $G_{1}$ by amalgamating $G_{2}$
with a free group over a $\integers$ subgroup, followed HNN-extensions over $\integers$ subgroups.  It follows from the results of \cite{Osi06Memoirs} that these $\integers$ subgroups
are maximal parabolic subgroups, hence we may apply Theorem~0.1 (3), (3').
\end{proof}

\begin{remark}\label{Rem:ParabolicsAreComputable}
From the Combination Theorem it follows that $G$ has finitely many maximal non-cyclic abelian subgroups up to conjugation, and we can construct, by induction, the list of them
along with a finite generating set for each.  In the base group $\Gamma$ this is possible using the results of \cite{Dah08}.
\end{remark}

We now recall in more detail the construction from \cite{KM05Implicit} of a Makanin-Razborov diagram encoding all solutions to a system of equations over a free group.
Let $S$ be an $F$-NTQ system.  To each
group $G_{i}$ is associated a set of \emph{canonical automorphisms}, based on the graph of groups decomposition of $G_{i}$.  See  \S 3.14 of \cite{KM06}
for the definition of canonical automorphisms.  All canonical
automorphisms fix $G_{i+1}$ pointwise.

A homomorphism $\pi_{i}:G_{i}\rightarrow G_{i+1}$ is a solution to $S_{i}$ over $G_{i+1}$.  To each sequence $\Pi=(\pi_{1},\dots,\pi_{n})$ of solutions
is associated a set of homomorphisms
\begin{equation}
\Phi(\Pi) = \{ \sigma_{1}\pi_{1}\ldots\sigma_{n}\pi_{n} \; | \; \mbox{$\sigma_{i}$ a canonical automorphism of $G_{i}$}\}
\end{equation}
called the \emph{fundamental sequence} associated with the NTQ system $S(X,A)$ and the sequence $\Pi$.  The following lemma is
an immediate corollary of the definition of canonical automorphism.

\begin{lemma}\label{Lem:QAbelian}
Let $S_{i}$ have the form (\ref{NTQ1}) and let $Q$ be the QH-subgroup associated with $S_{i}$.
Then for every solution $\pi_{i}$ of $S_{i}$ and canonical automorphism $\sigma$, if $Q^{\pi_{i}}$ is abelian then $Q^{\sigma\pi_{i}}$ is abelian.
\end{lemma}

The complete set of solutions to the system $S(X,A)$ is described by NTQ groups and fundamental sequences as follows.

\begin{proposition}\cite{KM98b}\label{Prop:FreeNTQ}
There is an algorithm that, given a system of equations $S(X,A)=1$ over a free group, produces
finitely many $F$-NTQ systems
\[
S_{1}(X_{1},A)=1, S_{2}(X_{2},A)=1, \ldots, S_{n}(X_{n},A)=1,
\]
along with, for each system $S_{i}(X_{i},A)$, a tuple of solutions $\Pi_{i}$ and a homomorphism $\mu_{i}: F(X,A)\rightarrow F_{S_{i}}$
such that
for every $F$-homomorphism $\psi:F_{S}\rightarrow F$ there is an index $i$ and a
homomorphism $\phi$ in the fundamental sequence $\Phi(\Pi_{i})$ such that
\[
\psi = \mu_{i}\phi.
\]
\end{proposition}


\subsection{Embedding $\Gamma$-NTQ into extensions of centralizers}
In this section we prove that $\Gamma$-NTQ groups, with a few exceptions, embed into groups obtained from $\Gamma$ by
extensions of centralizers.  When converting $F$-NTQ systems to $\Gamma$-NTQ systems in the next section, we will need to ensure that
the systems are not of the exceptional type.
Our results apply in the more general case when $\Gamma$ is any toral relatively hyperbolic group.

We will need the following lemma in order to construct embeddings inductively.

\begin{lemma}\label{Lem:CanonicalEmbedding}
Let $H\leq G$ be any torsion-free groups and let $S(X,A)$ be a system of equations over $H$.
\begin{romanenumerate}
\item If $S$ has one of the NTQ forms (\ref{NTQ1}), (\ref{NTQ3}), or (\ref{NTQ4}) then
the canonical homomorphism $H_{S} \rightarrow G_{S}$ is an embedding.
\item If $S$ has the form (\ref{NTQ2}), and $U'$ is a generating set for $C_{G}(U)$,
then the canonical homomorphism
\[
H_{S} \rightarrow \langle G, X\gst [x,y]=1\cs [x,u]=1\cs x,y\in X\cs u\in U'\rangle
\]
is an embedding.
\end{romanenumerate}
\end{lemma}
\begin{proof}
Form (\ref{NTQ1}) (when $S$ is a standard quadratic equation) is Proposition~2 of \cite{KM98a}, form (\ref{NTQ2}) ($H_{S}$ is an extension of a centralizer) follows
from the theory of normal forms for HNN-extensions, and forms (\ref{NTQ3}) and (\ref{NTQ4}) ($H_{S}$ is the free product of $H$ with a free group or free abelian group) are obvious.
\end{proof}

We first prove the embedding for NTQ systems of depth one, then extend by induction to systems of arbitrary depth.

\begin{lemma}\label{Lem:NTQEmbedding}
Let $G$ be a non-abelian toral relatively hyperbolic group generated by $B$ and let $S(X,B)$  be an NTQ system of depth one over $G$.
If $S(X,B)$ has the form (\ref{NTQ1}), we assume further that either
\begin{romanenumerate}
\item there exists a solution $\pi$ to $S(X,B)$ over $G$ such that $Q^{\pi}$ is non-abelian, where
$Q$ is the QH-subgroup of $G_{S}$ associated with $S$, or
\item $S$ has one of the forms $[x,y]$, $[x,y]d$, or $[x_{1},y_{1}][x_{2},y_{2}]$.
\end{romanenumerate}
Then the group $G_{S}=\langle G, X \, | \, S\rangle$ embeds into a group $G'$ which is an iterated centralizer extension of $G$, and there
is an algorithm to construct the emebedding.
In particular, $G_{S}$ is fully residually $G$ and $R_{G}(S) = \nclofin{S}{G[X]}$.
\end{lemma}
\begin{definition} When an NTQ system $S$ of depth one satisfies conditions (i) and (ii) we will call $S$ \emph{embeddable} over $G$.

An NTQ system $S=S_{1}\cup\ldots\cup S_{n}$ of depth $n$
will be called \emph{embeddable} over $G$ if each $S_{i}$ is embeddable over
$\langle G, X_{i+1},\ldots, X_{n}\gst S_{i+1},\ldots,S_{n}\rangle$.
\end{definition}
\begin{proof}
If $S$ has form (\ref{NTQ2}), then by assumption $U$ generates a centralizer in $G$, hence $\langle G, X\gst S\rangle$ is itself an extension of a centralizer of
$G$ so there is nothing to prove.
We consider the remaining forms (\ref{NTQ1}), (\ref{NTQ3}), and (\ref{NTQ4}).


\textbf{Form (\ref{NTQ1}): Quadratic equation}.  Suppose that $S$ is a quadratic equation, which we may assume has one of the standard forms
(\ref{Eqn:st1})--(\ref{Eqn:st4}). Equations of the form $[x,y]$ are a case of form (\ref{NTQ3}), which are covered below.

Most of the cases are covered by Theorem~4.1 of \cite{KM09}, which proves that the required
embedding exists when $S$ has the form $[x,y]d$ or
$[x_{1},y_{1}][x_{2},y_{2}]$, or when $\chi(S)\leq -2$ and $S$ has a `non-commutative' solution over $G$.  A non-commutative solution $\pi$ is one in which
there exist two adjacent atoms of the form $[x_{i}, y_{i}]$, $x_{i}^{2}$, $c_{i}^{z_{i}}$, or $d$ whose images under $\pi$ do not commute.

For orientable equations, it follows from Proposition~4.3 of \cite{KM09} that such a solution always exists.  For non-orientable and genus 0 forms,
we use the assumption that $Q^{\pi}$ is non-abelian for some solution $\pi$.  In the genus 0 case, there must exist $i$ and $j$ such
that the images of $c_{i}^{z_{i}}$ and $c_{j}^{z_{j}}$ do not commute.  Applying commutation-transitivity, we may assume that
$c_{i}^{z_{i}}$ and $c_{j}^{z_{j}}$ are adjacent hence $\pi$
is a non-commutative solution.  Similar reasoning applies to the non-orientable case.

The construction given in Theorem~4.1 of \cite{KM09} is explicit, and the reader may verify that to carry out the
construction effectively it suffices to be able to solve the word problem in $G$ and to have a solution to $S$ in `general position': this is
a solution in which the images of no two adjacent atoms commute.  The existence of such a solution follows from the existence of a non-commutative solution
(see Proposition~3 of \cite{KM98a}), and one may be produced by exhaustive search.

It remains to consider forms with $\chi(S)>-2$.  The forms $x^{2}$, $x^{2}d$, $x^{2}y^{2}$, and $c^{z}d$ need not be considered since the image of $Q$ under any solution
must be abelian.  Forms $[x,y]$ and $[x,y]d$ were covered above. The remaining forms are
$c_{1}^{z_{1}} c_{2}^{z_{2}}$, $x^{2}c^{z}d$, $x^{2}y^{2}$, $x^{2}y^{2}d$, $x^{2} y^{2} z^{2}$.  As discussed above, the solution $\pi$ is non-commutative,
and such a solution may be produced by exhaustive search since it is assumed to exist.

The techniques for embedding the group $\langle G,X\gst S\rangle$
into extensions of centralizers in these cases have been discussed in detail, for the case when $G$ is a limit group, in \S 5 of \cite{KM98a}.
The proofs in our case are similar, so we omit some details.

For the form $x^{2}y^{2}d$, consider the series of extensions of centralizers
\begin{eqnarray*}
G' & = & \langle G, t \gst [C(ab), t] \rangle, \\
G'' & = & \langle G', s \gst [C(atat), s] \rangle, \\
G''' & = & \langle G'', r \gst [C(s^{-1}atst^{-1}b), r] \rangle,
\end{eqnarray*}
and the map $\psi: \langle G, x, y \gst x^{2} y^{2} d \rangle \rightarrow G'''$ given by
\begin{eqnarray*}
x & \rightarrow & (at)^{s}r \\
y & \rightarrow & r^{-1} t^{-1} b.
\end{eqnarray*}
Since $(x^{2} y^{2} d)^{\psi} = 1$, $\psi$ defines a homomorphism.  Using normal forms for elements of HNN-extensions, one can show that
$\psi$ is injective. 

For the form $x^{2} c^{z} d$, consider the sequence of extensions of centralizers
\begin{eqnarray*}
G' & = & \langle G, t \gst [C(d), t] \rangle, \\
G'' & = & \langle G', s \gst [C(c^{b}), s] \rangle, \\
G''' & = & \langle G'', r \gst [C(c^{bt}), r] \rangle,
\end{eqnarray*}
and the map $\psi: \langle G, x, y \gst x^{2} c^{z} d \rangle \rightarrow G'''$ given by
\begin{eqnarray*}
x & \rightarrow & a^{t} \\
y & \rightarrow & bstr.
\end{eqnarray*}
As in the previous case, $(x^{2} c^{z} d)^{\psi} = 1$ and we may prove using normal forms that $\psi$ is injective.

For the form $x^{2} y^{2} z^{2}$,
find any solution $x\rightarrow a$, $y\rightarrow b$, $z\rightarrow c$ of $S$ in general position.  Consider the series of six
extensions of centralizers
\begin{eqnarray*}
G^{(1)} & = & \langle G, s \gst [s, C(ab)] \rangle, \\
G^{(2)} & = & \langle G^{(1)}, r \gst [r, C(s^{-1}bc)] \rangle, \\
G^{(3)} & = & \langle G^{(2)}, v \gst [v, C(abrs^{-1}bc)] \rangle, \\
G^{(4)} & = & \langle G^{(3)}, t \gst [t, C(vasvas)] \rangle, \\
G^{(5)} & = & \langle G^{(4)}, u \gst [u, C(s^{-1}brs^{-1}br)] \rangle, \\
G^{(6)} & = & \langle G^{(5)}, w \gst [w, C(r^{-1}cv^{-1}r^{-1}cv^{-1})] \rangle,
\end{eqnarray*}
and the map $\psi: \langle G, x, y, z \gst x^{2} y^{2} z^{2} \rangle \rightarrow G^{(6)}$ given by
\begin{eqnarray*}
x & \rightarrow & (vas)^{t} \\
y & \rightarrow & (s^{-1}br)^{u} \\
z & \rightarrow & (r^{-1}cv^{-1})^{w}.
\end{eqnarray*}
As in the previous case, $(x^{2} y^{2} z^{2})^{\psi} = 1$ and we may prove, with a lengthy argument using normal forms, that $\psi$ is injective.

For the form $c_{1}^{z_{1}}c_{2}^{z_{2}}d$, we may follow the proof given in section \S 5 case 2 of \cite{KM98a}.


\textbf{Form (\ref{NTQ4}): Free product with a free group.}
First consider the case of a single variable, $X=\{x\}$.  We will apply the following fact.
\begin{lemma}[\cite{KM98a}, Lemma 16]\label{Lem:AbelianEmbeds}
Let $H$ be a non-abelian CSA group, $H'=\langle H, t\gst [t, C(u)]\rangle$ an extension of centralizer of $H$, and $v\in H$ an element such that
$[u,v]\neq 1$.  Then for $z=(tvt)^{2}$ the subgroup $\langle G, u^{z}, t^{z}\rangle$ is $H$-isomorphic to the free product of $H$ and a free abelian group of
rank two generated by $u^{z}$ and $t^{z}$.
\end{lemma}
Construct two non-commuting elements $u, v\in G$ and let $G'=\langle G, t\gst [t, C(u)]\rangle$.  By the lemma, $G$ embeds into $G'$ via $x\rightarrow u^{z}$.

When $X=\{x_{1},\ldots,x_{n}\}$, we first embed $G\ast F(x_{1})$ into an extension of centralizer $G'$ as above.  Then by Lemma~\ref{Lem:CanonicalEmbedding},
$G\ast F(x_{1},x_{2})$ embeds canonically in $G'\ast F(x_{2})$.  We may now repeat the construction above and proceed by induction.


\textbf{Form (\ref{NTQ3}): Free product with a free abelian group}.  Suppose $S$ has the form (\ref{NTQ3}).
First, suppose that $|X|=2$, and so
$\langle G, X \gst S\rangle\simeq G \ast \integers^{2}$.
By Lemma~\ref{Lem:AbelianEmbeds}, $G \ast \integers^{2}$ embeds effectively into a centralizer extension of $G$.

If $X=\{x_{1},\ldots, x_{n}\}$ with $n>2$, partition $X$ into $X_{a}=\{x_{3},\ldots,x_{n}\}$ and $X_{b}=\{x_{1},x_{2}\}$ and partition $S$ into
two subsystems
\begin{eqnarray*}
S_{a} & = & \{ [x_{i},x_{j}]=1, [x_{i},u]=1, \sst i,j\in \{3,\ldots, m\}\cs u\in X_{b} \}, \\
S_{b} & = & \{ [x_{1},x_{2}]=1 \}.
\end{eqnarray*}
The system $S_{b}$ has the form (\ref{NTQ3}) with two variables, hence $\langle G, X_{b}\gst S_{b}\rangle$ embeds effectively into an extension of centralizers $G'$
of $G$.  Using Lemma~\ref{Lem:AlgorithmsRelativelyHyperbolic}, construct a finite generating set $U$ for the centralizer of $\{x_{1},x_{2}\}$ in $G'$.
By Lemma~\ref{Lem:CanonicalEmbedding},  $\langle G, X_{b}, X_{a}\gst S_{b}, S_{a}\rangle$ embeds canonically into
$\langle G', X_{a}\gst [x,y]=1, [x,u]=1\cs x,y\in X_{a}\cs u\in U'\rangle$, which is an iterated centralizer extension of $G$.
\end{proof}

Finally, we combine the two previous lemmas to obtain an effective embedding of $\Gamma$-NTQ groups into extensions of centralizers of $\Gamma$.

\begin{corollary}\label{Cor:NTQEmbedding}
Let $S(X,A)$ be an embeddable NTQ system over a toral relatively hyperbolic group $\Gamma$.
Then $\Gamma_{S}=\langle \Gamma, X\gst S\rangle$ embeds into a group obtained from $\Gamma$ by
extensions of centralizers.  In particular, $\Gamma_{S}$ is fully residually $\Gamma$, is toral relatively hyperbolic, and
\[
\Gamma_{S} = \Gamma_{R(S)}.
\]
Further, the embedding may be computed effectively.
\end{corollary}
\begin{proof}
We proceed by induction of the depth of the NTQ system.  For depth 0, there is nothing to prove.
Assume then that for the system $S'=S_{2}\cup\ldots\cup S_{n}$ the group $\Gamma_{S'}$ embeds effectively into a group $G$ obtained from
$\Gamma$ by extensions of centralizers.  Then consider $S_{1}$ as a system of equations over $G\geq \Gamma_{S'}$.

If $S_{1}$ has the form (\ref{NTQ2}), we construct, using Lemma~\ref{Lem:AlgorithmsRelativelyHyperbolic}, a generating set $U'$
for the centralizer of $\langle U\rangle$ in $G$. Then $\langle \Gamma_{S'}, X_{1}\gst S_{1}\rangle$ embeds canonically
into the centralizer extension of $G$ given in Lemma~\ref{Lem:CanonicalEmbedding}. Otherwise, Lemma~\ref{Lem:CanonicalEmbedding} embeds $\langle \Gamma_{S'}, X_{1}\gst S_{1}\rangle$ into
$G_{S_{1}}$, which is further embedding into a centralizer extension of $G$ using Lemma~\ref{Lem:NTQEmbedding}.

It is proved in \cite{KM09} that extensions of centralizers of $\Gamma$ are discriminated by $\Gamma$, hence so are their subgroups, so
$\Gamma_{S}$ is discriminated by $\Gamma$.  It follows that $\Gamma_{S}=\Gamma_{R(S)}$, so Lemma~\ref{Lem:PropertiesOfGammaNTQ}
proves that $\Gamma_{S}$ is toral relatively hyperbolic.
\end{proof}


\subsection{Converting $F$-NTQ systems to $\Gamma$-NTQ systems}\label{section:ConvertingNTQ}

Let $S(X,A)$ be an $F$-NTQ system and $\Pi=(\pi_{1},\ldots,\pi_{n})$ a solution sequence.
We may regard $S(X,A)$ as a system of equations over $\Gamma$, but the
conditions for Corollary~\ref{Cor:NTQEmbedding} need not hold due to degeneracy in the NTQ structure.  The goal of this section
is to produce a $\Gamma$-NTQ system $T$ satisfying Corollary~\ref{Cor:NTQEmbedding} and such that each homomorphism
from $\Phi(\Pi)$, after post-composition with the canonical projection $\pi: F\rightarrow \Gamma$, factors through $\Gamma_{T}$.

The system of equations $T$ will be obtained by adding relators to $S$, and we denote by $\gamma:F_{S}\rightarrow \Gamma_{T}$ the
canonical quotient map.  These relators will be consequences of the solutions in
the fundamental sequence (they are in the radical $S$ with respect to this set of homomorphisms), and as such the full set of
homomorphisms in the fundamental sequence will factor through $\Gamma_{T}$, as we will see in Lemma~\ref{Lem:ConvertedNTQ}.
The system $T$ will not be written in NTQ form, but Lemma~\ref{Lem:GammaTisNTQ} will establish that $\Gamma_{T}$ is (isomorphic to) a
$\Gamma$-NTQ group.

We proceed inductively on the depth $n$ of the system $S$.  A depth 0 system has no equations and no variables, so $F_{S}=F$.  Set $T$ to be
the set of relators of $\Gamma$ and obtain $\Gamma_{T}=\Gamma$ and $\gamma=\pi$.

Now assume that for the system $S'=S_{2}\cup\ldots\cup S_{n}$ in variables $X'=X_{2}\cup\ldots\cup X_{n}$ and solution sequence
$\Pi'=(\pi_{2},\ldots,\pi_{n})$ we have constructed
a quotient $\Gamma_{T'}$ of $F_{S'}$, and let $\gamma'$ denote the quotient map.
Note that since Lemma~\ref{Lem:GammaTisNTQ} is proved inductively, we may assume that $\Gamma_{T'}$ is an embeddable $\Gamma$-NTQ
group.  In particular, $\Gamma_{T'}$ is toral relatively hyperbolic.

We form $\Gamma_{T}$ by adding variables $X_{1}$ new relators $T''$.  Denote $T=T'\cup T''$.
In every case, $S_{1}$ will be included in $T''$, ensuring that $\Gamma_{T}$ is a quotient of $F_{S}$. We proceed
based on the NTQ form (\ref{NTQ1})--(\ref{NTQ4}) of $S_{1}$.

\emph{Forms (\ref{NTQ2}), (\ref{NTQ3}), and (\ref{NTQ4}).}
If $S_{1}$ has form (\ref{NTQ3}) or (\ref{NTQ4}), set $T''=\{S_{1}\}$.
If $S_{1}$ has form (\ref{NTQ2}), there are two cases.
 If $u^{\gamma'}=1$ for all $u\in U$, set $T''=\{S_{1}\}$.
Otherwise, let $u_{0}\in U$ with $u_{0}^{\gamma'}\neq 1$.  Using
Lemma~\ref{Lem:AlgorithmsRelativelyHyperbolic}, construct a generating set $U'$ for the centralizer of $u_{0}^{\gamma'}$ in $\Gamma_{T'}$ and let $T''$ consist of
$S_{1}$ and all relators $[x, u']$ where $x\in X_{1}$ and $u'\in U'$.

\emph{Form (\ref{NTQ4}).}
If $S_{1}$ has form (\ref{NTQ1}), we may assume it is a single quadratic equation having one of the standard forms
(\ref{Eqn:st1})--(\ref{Eqn:st4}). Consider the equation $S_{1}^{\gamma'}$ over $\Gamma_{T'}$.
The constants $c_{i}$ and $d$ in the standard form are non-trivial in $F_{S'}$, but their images $c_{i}^{\gamma'}$, $d^{\gamma'}$ may be trivial in $\Gamma_{T'}$.
Form an equation $S_{1}'$ over $\Gamma_{T'}$
by
\begin{romanenumerate}
\item erasing from $S_{1}$ each factor $c_{i}^{z_{i}}$ such that $c_{i}^{\gamma'}=1$, and
\item if $d^{\gamma'}=1$,  erasing $d$ and replacing the rightmost factor of the form $c_{i}^{z_{i}}$ by $c_{i}$.
\end{romanenumerate}
In $\Gamma_{T'}$, the relator $S_{1}$ follows from $S_{1}'$.  Consequently, if we include $S_{1}'$ in $T''$ then $S_{1}$ is redundant
and need not be included.

Let $X_{1}'$ denote the set of variables appearing in $S_{1}'$ and $X_{1}''=X_{1}\setminus X_{1}'$.
Since $\pi_{1}$ is a solution to $S_{1}$ over $F_{S'}$, and $\Gamma_{T'}$ is a quotient of $F_{S'}$,
it follows from the construction of $S_{1}'$ that $\pi_{1}'=\pi_{1}\gamma'$ is a solution to $S_{1}'$ over $\Gamma_{T'}$.

\[
\xymatrix{
&  F_{S}  \ar[d]^{\pi_{1}} \ar[dl]_{\pi_{1}'} \\
\Gamma_{T'} & F_{S'} \ar[l]^{\gamma'}
}
\]

Let $Q\leq F_{S}$ be the QH-subgroup associated with $S_{1}$, i.e. the group generated by
$W=\{x_{1},y_{1},\ldots,x_{n}, y_{n}, c_{1}^{z_{1}}, \ldots, c_{m}^{z_{m}}, d\}$ in the orientable case
and $W=\{x_{1},\ldots,x_{n}, c_{1}^{z_{1}}, \ldots, c_{m}^{z_{m}}, d\}$ in the non-orientable case.
We check whether or not the image $Q^{\pi_{1}'}$
is an abelian subgroup in $\Gamma_{T'}$, by verifying that the images of all commutators of elements of $W$ are trivial under $\pi_{1}'$.
We proceed based on the form of $S_{1}'$ and whether or not $Q^{\pi_{1}'}$ is abelian.

\emph{Case $\chi(S_{1}')\leq -2$.}
If $Q^{\pi_{1}'}$ is non-abelian, then $T''=\{ S_{1}'\}$.
For any $h\in \Gamma_{T'}$,  let $U(h)$ denote a generating set for the centralizer of $h$ in $\Gamma_{T'}$,
which we may construct using Lemma~\ref{Lem:AlgorithmsRelativelyHyperbolic}.
We denote $N=\{1,\ldots, n\}$, $M=\{1,\ldots,m\}$, $a_{i}=z_{i}^{\pi_{1}'}$,
\[
XY=\{ [x_{i}, y_{j}]\cs [x_{i}, x_{j}]\cs [y_{i}, y_{j}]\sst i,j \in N\}\cs X = \{ [x_{i}, x_{j}] \sst i,j\in N\},
\]
and
\[
Z = \{ [a_{1}a_{i}^{-1}z_{i} a_{1}^{-1}, a_{1}a_{j}^{-1}z_{j}a_{1}^{-1}]\cs [a_{1}a_{i}^{-1}z_{i}a_{1}^{-1}, u] \sst i,j\in M\cs u\in U(c_{1}) \}.
\]
If $Q^{\pi_{1}'}$ is abelian, then the construction of $T''$ is given in the table below.

\renewcommand{\arraystretch}{1.3}
\begin{center}
\begin{tabular}{ l | l }
\textbf{Form of $\mathbf{S_{1}'}$}  & $\mathbf{T''}$ \\ \hline
$\prod_{i=1}^{n}[x_{i},y_{i}] $ & $\{ S_{1}'\} \cup XY$ \\ \hline
$\prod_{i=1}^{n}[x_{i},y_{i}] \prod_{i=1}^{m} c_{i}^{z_{i}} d$  & $\{S_{1}'\}\cup XY\cup Z$ \\ \hline
$\prod_{i=1}^{n}x_{i}^2$ &  $\{ S_{1}'\cs x_{1}\ldots x_{n}\}\cup X$ \\ \hline
$\prod_{i=1}^{n}x_{i}^2 \prod_{i=1}^{m} c_{i}^{z_{i}} d$  & $\{ S_{1}', x_{1}\ldots x_{n}\} \cup X\cup  Z$ \\ \hline
$\prod_{i=1}^{m} c_{i}^{z_{i}} d$ & $\{ S_{1}'\} \cup Z$ \\
\end{tabular}
\end{center}

\emph{Case $\chi(S_{1}')>-2$.}  Denote $a=x^{\pi_{1}'}$, $b=y^{\pi_{1}'}$, and $g=z^{\pi_{1}'}$.
The construction of $T''$ is given in the table below, and depends in some cases on whether  $Q^{\pi_{1}'}$ is
abelian or non-abelian.


\begin{center}
\begin{tabular}{l | l | l}
\textbf{Form of $\mathbf{S_{1}'}$} & \textbf{Condition on } $\mathbf{Q^{\pi_{1}'}}$ & $\mathbf{T''}$ \\ \hline
$x^{2}$ & no condition & $\{x^{2}, x\}$ \\ \hline
$x^{2}d$ & no condition & $\{x^{2}d, xa^{-1}\}$ \\ \hline
$c^{z}d$ & no condition & $\{c^{z}d, [zg^{-1}, u]\sst u\in U(c)\}$ \\ \hline
$[x,y]$ & no condition & $\{[x,y]\}$ \\ \hline
$[x,y]d$ & no condition & $\{ [x,y]d\}$ \\ \hline
$x^{2}y^{2}$ & no condition & $\{x^{2}y^{2}, xy\}$  \\ \hline
$x^{2}y^{2}d$ & non-abelian & $\{x^{2}y^{2}d\}$ \\ \hline
$x^{2}y^{2}d$ & abelian & $\{x^{2}y^{2}d, xyb^{-1}a^{-1}, [x,u]\gst u\in U(ab)\}$ \\ \hline
$x^{2}y^{2}z^{2}$ & non-abelian & $\{x^{2}y^{2}z^{2}\}$ \\ \hline
$x^{2}y^{2}z^{2}$ & abelian & $\{x^{2}y^{2}z^{2}, [x,y], xyz\}$\\ \hline
$x^{2}c^{z}d$ & non-abelian & $\{x^{2}c^{z}d\}$ \\ \hline
$x^{2}c^{z}d$ & abelian & $\{x^{2}c^{z}d, xa^{-1}, [zg^{-1},u]\gst u\in U(c)\}$\\ \hline
$c_{1}^{z_{1}}c_{2}^{z_{2}}d$ & abelian & $\{c_{1}^{z_{1}}c_{2}^{z_{2}}d\}\cup Z$ \\ \hline
$c_{1}^{z_{1}}c_{2}^{z_{2}}d$ & non-abelian & $\{c_{1}^{z_{1}}c_{2}^{z_{2}}d\}$
\end{tabular}
\end{center}
\vspace{1em}

This completes the construction of $T$.  We now prove that $T$ has the desired properties.

\begin{lemma}\label{Lem:GammaTisNTQ}
The group $\Gamma_{T}$ is isomorphic to an embeddable $\Gamma$-NTQ group.  Further, the isomorphism
can be constructed effectively.
\end{lemma}
\begin{proof}
We proceed by induction on the depth of the system $S$.  For depth 0, there is nothing to prove.  Assume the statement holds for $\Gamma_{T'}$.
If $S_{1}$ has the NTQ form (\ref{NTQ3}) or (\ref{NTQ4}), then $\Gamma_{T}$ is already NTQ.

Suppose $S_{1}$ has the form (\ref{NTQ2}).  If all $u\in U$ are trivial in $\Gamma_{T'}$, the relators $[x,u]$ are trivial and $T$ is equivalent to a system of the form (\ref{NTQ3}).
Otherwise, the construction guarantees that $U'$ generates a centralizer in $\Gamma_{T'}$ hence $T$ is a system of the form (\ref{NTQ2}).

Now suppose that $S_{1}$ has the form (\ref{NTQ1}).
The variables $X_{1}''$ do not appear in $S_{1}'$, so $\Gamma_{T}\simeq \langle \Gamma_{T'}, X_{1}' \gst T''\rangle \ast F(X_{1}'')$, which is an NTQ-extension of $\langle \Gamma_{T'}, X_{1}' \gst T''\rangle$
of the form (\ref{NTQ4}).  So it suffices to consider $\langle \Gamma_{T'}, X_{1}' \gst T''\rangle$.

\emph{Cases with $\chi(S_{1}')\leq -2$.} When $Q^{\pi_{1}'}$ is non-abelian, then $\Gamma_{T}$ is an embeddable NTQ system.
Now assume that $Q^{\pi_{1}'}$ is abelian.

Since $\Gamma_{T'}$ is commutation-transitive, $[c_{1}^{a_{1}},c_{i}^{a_{i}}]=1$, hence $c_{i}^{a_{i}a_{1}^{-1}}\in C(c_{1})$ for all $i\in M$.
Consequently,
the relation
\[
[a_{1}a_{i}^{-1}z_{i}a_{1}^{-1}, c_{i}^{a_{i}a_{1}^{-1}}]=1
\]
follows from $Z$.  Conjugating by $a_{1}a_{i}^{-1}$ we obtain $[z_{i}a_{i}^{-1}, c_{i}]=1$ which may be rewritten as $c_{i}^{z_{i}} = c_{i}^{a_{i}}$.  Hence
$c_{1}^{z_{1}}\cdots c_{m}^{z_{m}} d= c_{1}^{a_{1}}\cdot c_{m}^{a_{m}} d$ follows from $Z$.
From $XY$, it  follows that $\prod_{i=1}^{n} [x_{i}, y_{i}]=1$ and from $X$ and the relator $x_{1}\ldots x_{n}$ it follows that $\prod_{i=1}^{n} x_{i}^{2}=1$.

Then in all cases, the relator $S_{1}'$ follows from the other relators and may be eliminated.  In the genus 0 case,
we can realize $\Gamma_{T}$ as an NTQ extension of the form (\ref{NTQ2}) via the change-of-variables $t_{i}=a_{1}a_{i}^{-1}z_{i}a_{1}^{-1}$,
and in the other cases $\Gamma_{T}$ is a free product with a free abelian group followed by an extension of a centralizer.

\emph{Forms $x^{2}$ and $x^{2}d$.}  Since $\langle \Gamma_{T'}, X_{1}' \gst T''\rangle\simeq \Gamma_{T'}$
via the obvious isomorphism, there is nothing to prove.

\emph{Form $c^{z}d$.}  We will show that the relator $c^{z}d$ follows from the other relators, hence can be eliminated.  Then via the isomorphism $z\rightarrow tg$,
$\Gamma_{T}\simeq \langle \Gamma_{T'}, t\gst [t,u]\cs u\in U\rangle$, which has the form (\ref{NTQ2}) with $U$ generating a centralizer as required.
We have $[zg^{-1}, c]=1$, which is equivalent to $c^{z}=c^{g}$.  Since $g$ is the image of $z$ under the solution $\pi_{1}'$, $c^{g}=d^{-1}$ hence
we have $c^{z}d=1$.



\emph{Form $x^{2}y^{2}$.} In this case, $\langle \Gamma_{T'}, X_{1}' \gst T''\rangle\simeq \Gamma_{T'}\ast \langle x\rangle$ and this group
has the NTQ form (\ref{NTQ4}).

\emph{Forms $x^{2}y^{2}d$, $x^{2}y^{2}z^{2}$, and $x^{2}c^{z}d$ with $Q^{\pi_{1}'}$ non-abelian in each case, and forms $[x,y]$ and $[x,y]d$.}
In these cases, $\langle \Gamma_{T'}, X_{1}' \gst T''\rangle$ has the embeddable NTQ form (\ref{NTQ1}).

\emph{Form $x^{2}y^{2}d$ with $Q^{\pi_{1}'}$ abelian.} In $\langle \Gamma_{T'}, X_{1}' \gst T''\rangle$ we have that $y=abx^{-1}$.  Since $x$ commutes with $ab$, we have
$x^{2}y^{2}d=x^{2}(abx^{-1})^2d = a^{2}b^{2}d = 1$.  We eliminate the relators $xyb^{-1}a^{-1}$ and $x^{2}y^{2}d$ and see that $\langle \Gamma_{T'}, X_{1}' \gst T''\rangle$ has the form (\ref{NTQ2}).

\emph{Form $x^{2}y^{2}z^{2}$ with $Q^{\pi_{1}'}$ abelian.} Since $z=y^{-1}x^{-1}$ and $x$ and $y$ commute, the first relator is trivial and we may eliminate $z$
using a Tietze transformation obtaining the form (\ref{NTQ3}).

\emph{Form $x^{2}c^{z}d$ with $Q^{\pi_{1}'}$ abelian.}  Since $x=a$, we elmininate $x$ using a Tietze transformation.
The relator $x^{2}c^{z}d$ reduces to $c^{z}da^{-2}$, so we may apply the argument for the form $c^{z}d$.

\emph{Form $c_{1}^{z_{1}}c_{2}^{z_{2}}d$.} The same reasoning from the case $\chi(S_{1}')\leq -2$ applies.

\end{proof}

\begin{lemma}\label{Lem:ConvertedNTQ}
The map $\pi_{1}'$ is a solution to $T''$ over $\Gamma_{T'}$ and
for every homomorphism $\phi\in\Phi(\Pi)$, the composition $\phi\pi: F_{S}\rightarrow \Gamma$ factors through $\gamma$.

\end{lemma}
\begin{proof}

Proceed by induction of the depth of the NTQ system $S$. For depth 0, there is nothing to prove. Now
assume that the statements hold for $F_{S'}$ and $\Gamma_{T'}$.

We begin by formulating the conditions that we must check for each case of the construction.
To check that $\pi_{1}'$ is a solution to $T''$ we must ensure that
\begin{equation}\label{Cond1}
R^{\pi_{1}'}=1,
\end{equation}
for every relator $R$ of $T''$.

To prove that $\phi\pi$ factors through $\gamma$, let $\phi$ be given by $\phi=\sigma_{1}\pi_{1}\phi_{2}$ where $\phi_{2}\in\Phi(\Pi')$
and $\sigma_{1}$ is a canonical automorphism.  By induction, $\phi_{2}\pi$ factors through
$\gamma'$.  Then it suffices to prove that $\sigma_{1}\pi_{1}\gamma':F_{S}\rightarrow\Gamma_{T'}$ factors through $\gamma$.  Define
$\psi: \Gamma_{T}\rightarrow \Gamma_{T'}$ as follows.  Since $\gamma$ is surjective, every element $w'$ of $\Gamma_{T}$ has a
preimage $w$ in $F_{S}$.  Then we define
\[
(w')^{\psi} = w^{\sigma_{1}\pi_{1}\gamma'}.
\]
Provided $\psi$ is well-defined, we have that $\gamma\psi=\sigma_{1}\pi_{1}\gamma'$ by definition hence $\sigma_{1}\pi_{1}\gamma'$ factors through $\gamma$.
To check that $\psi$ is well-defined, it suffices to prove that each relator $R$ of $T''$, viewed as an element of $F_{S}$, satisfies
\begin{equation}\label{Cond2}
R^{\sigma_{1}\pi_{1}\gamma'}=1.
\end{equation}

Since (\ref{Cond1}) is the special case $\sigma_{1}=\mathrm{id}$ of (\ref{Cond2}), it suffices to prove (\ref{Cond2}) for each case of the construction.
Since $\sigma_{1}$ is an $F_{S'}$-automorphism of $F_{S}$, $\sigma_{1}\pi_{1}$ is a solution to $S_{1}$ over
$F_{S'}$ and it follows that $\sigma_{1}\pi_{1}\gamma'$ is a solution to $S_{1}$ and to $S_{1}'$ over $\Gamma_{T'}$. Hence we need only consider
cases of the construction that introduce relators other than $S_{1}$ and $S_{1}'$.  In particular, there is nothing to prove for the forms (\ref{NTQ3}) and
(\ref{NTQ4}).  We consider the remaining forms.

\emph{Form (\ref{NTQ2}).} If there exists $u_{0}\in U$ satisfying $u_{0}^{\gamma'}\neq 1$, then relators of the form
$[x,u]$ where $x\in X_{1}$ and $u\in C_{\Gamma_{T'}}(u_{0}^{\gamma'})$
are added in the construction.  Let $\bar{a}=x^{\sigma_{1}\pi_{1}\gamma'}$.  If $\bar{a}=1$ then $[x, u]^{\sigma_{1}\pi_{1}\gamma'}=1$. Otherwise,
since $[\bar{a},u_{0}^{\gamma'}]=1$ and $\Gamma_{T'}$ is commutation-transitive,
it follows that $[\bar{a},u]=1$ hence $[x,u]^{\sigma_{1}\pi_{1}\gamma'}=1$ as required.

\emph{Form (\ref{NTQ1}), with $\chi(S_{1}')\leq -2$ and $Q^{\pi_{1}'}$ abelian}. From Lemma~\ref{Lem:QAbelian}, $Q^{\sigma_{1}\pi_{1}\gamma'}$ is abelian, hence every
relator in the sets $XY$ and $X$ from the construction is sent to the identity under $\sigma_{1}\pi_{1}\gamma'$.
Let $\overline{a}_{i}=z_{i}^{\sigma_{1}\pi_{1}\gamma'}$.  In all cases,  we have
$c_{1}^{\overline{a}_{1}}\ldots c_{m}^{\overline{a}_{m}}d = 1 = c_{1}^{a_{1}}\ldots c_{m}^{a_{m}}d$.  From commutation-transitivity, all $c_{i}^{\overline{a}_{i}}$
commute and all of these commute with the non-trivial element $d$.
Similarly, all $c_{i}^{a_{i}}$ commute with each other and with $d$, so it
follows that $[c_{i}^{a_{i}}, c_{j}^{\overline{a}_{j}}]=1$ for all $i,j\in M$.

From $[c_{i}^{\overline{a}_{i}}, c_{i}^{a_{i}}]=1$ and the CSA property it follows that $[\bar{a}_{i}a_{i}^{-1}, c_{i}]=1$.  Conjugating by $a_{i}a_{1}^{-1}$ yields
\[
[a_{1}a_{i}^{-1} \bar{a}_{i} a_{1}^{-1}, c_{i}^{a_{i}a_{1}^{-1}}]=1.
\]
The equation $[c_{1}^{a_{1}}, c_{i}^{a_{i}}]=1$ gives $[c_{1}, c_{i}^{a_{i}a_{1}^{-1}}]=1$.  Then for any $u\in C(c_{1})$, transitivity implies that
$u$ commutes with $a_{1}a_{i}^{-1} \bar{a}_{i} a_{1}^{-1}$ hence $[a_{1}a_{i}^{-1} z_{i} a_{1}^{-1}, u]^{\sigma_{1}\pi_{1}\gamma'}=1$, for
all $i\in M$ and $u\in U(c_{1})$ as required.  Applying transitivity via $u$, we also obtain
$[a_{1}a_{i}^{-1} z_{i} a_{1}^{-1}, a_{1}a_{j}^{-1} z_{j} a_{1}^{-1}]^{\sigma_{1}\pi_{1}\gamma'}=1$.

\emph{Form (\ref{NTQ1}), with $\chi(S_{1}')>-2$.} Denote $\bar{a}=x^{\sigma_{1}\pi_{1}\gamma'}$, $\bar{b}=y^{\sigma_{1}\pi_{1}\gamma'}$, and $\bar{g}=z^{\sigma_{1}\pi_{1}\gamma'}$.

\emph{Form $x^{2}$.}  The additional relator is $x$. Since $\bar{a}^{2}=1$ and $\Gamma_{T'}$ is torsion-free, we have $\bar{a}=1$
hence $x^{\sigma_{1}\pi_{1}\gamma'}=1$ as required.

\emph{Form $x^{2}d$.}  The additional relator is $xa^{-1}$.  We have $\bar{a}^2=d^{-1}=a^{2}$, and $\bar{a}\neq 1$ since $d\neq 1$ and $\Gamma_{T'}$ is torsion-free,
so it follows from commutation-transitivity in $\Gamma_{T'}$ that $[\bar{a},a]=1$.
Since $\bar{a}^{2}a^{-2}=1$ and $\Gamma_{T'}$ is torsion-free, $\bar{a}=a$ hence $(xa^{-a})^{\sigma_{1}\pi_{1}\gamma'}=1$ as required.

\emph{Form $c^{z}d$.} The additional relators are $[zg^{-1}, u]$ for $u$ in $U(c)$.
Since $c^{\bar{g}}=d^{-1}=c^{g}$, we have $\bar{g}^{-1}c\bar{g}=g^{-1}c^{-1}g$ which may be rewritten as $[\bar{g}g^{-1}, c]=1$. By commutation-transitivity,
$\bar{g}g^{-1}$ commutes with $u$ hence $[zg^{-1}, u]^{\sigma_{1}\pi_{1}\gamma'}=1$ as required.

\emph{Form $x^{2}y^{2}$.} The additional relator is $xy$.  If one of $\bar{a}$, $\bar{b}$ is trivial, so is the other hence $(xy)^{\sigma_{1}\pi_{1}\gamma'}=1$.  Assume
$\bar{a}\neq 1$ and $\bar{b}\neq 1$.  Since $\bar{a}^{2}=\bar{b}^{-2}$, it follows from commutation-transitivity that $\bar{a}$ and $\bar{b}$ commute.  Then
$\bar{a}^{2}\bar{b}^{2}=1$ implies $\bar{a}\bar{b}=1$ hence $(xy)^{\sigma_{1}\pi_{1}\gamma'}=1$ as required.

\emph{Form $x^{2}y^{2}d$ with $Q^{\pi_{1}'}$ abelian.} The additional relators are $xyb^{-1}a^{-1}$ and $[x, u]$ where $u\in U(ab)$.
We have $\bar{a}^{2}\bar{b}^{2}=d^{-1}=a^{2}b^{2}$, hence $\bar{a}^{2}\bar{b}^{2}b^{-2}a^{-2}=1$.
From Lemma~\ref{Lem:QAbelian}, $Q^{\sigma_{1}\pi_{1}\gamma'}$ is abelian hence $\bar{a}$ and $\bar{b}$ commute.  It
follows that $(\bar{a}\bar{b}b^{-1}a^{-1})^{2}=1$ which implies $\bar{a}\bar{b}b^{-1}a^{-1}=1$ since $\Gamma_{T'}$ is torsion-free.  Hence
$(xyb^{-1}a^{-1})^{\sigma_{1}\pi_{1}\gamma'}=1$.

For the commutators, we have that $\bar{b}=\bar{a}^{-1}ab$ hence $\bar{a}^{2}(\bar{a}^{-1}ab\bar{a}^{-1}ab)d=1$.  Since $d=b^{-2}a^{-2}$ it follows that
$[\bar{a}, ab]=1$ and so $[\bar{a}, u]=1$ by transitivity of commutation.  Hence $[x,u]^{\sigma_{1}\pi_{1}\gamma'}=1$ as required.

\emph{Form $x^{2}y^{2}z^{2}$ with $Q^{\pi_{1}'}$ abelian.} The additional relators are $[x,y]$ and $xyz$. Since $Q^{\sigma_{1}\pi_{1}\gamma'}$ is abelian,
$[x,y]^{\sigma_{1}\pi_{1}\gamma'}=1$.  Since $\bar{a}^{2}\bar{b}^{2}\bar{c}^{2}=1$ and $\bar{a}$, $\bar{b}$, and $\bar{c}$ commute it follows from the fact that
$\Gamma_{T'}$ is torsion-free that $\bar{a}\bar{b}\bar{c}=1$ hence $(xyz)^{\sigma_{1}\pi_{1}\gamma'}=1$.

\emph{Form $x^{2}c^{z}d$ with $Q^{\pi_{1}'}$ abelian. } The additional relators are $xa^{-1}$ and $[zg^{-1}, u]$ for $u\in U(c)$.
Since $Q^{\pi_{1}\gamma'}$ and $Q^{\sigma_{1}\pi_{1}\gamma'}$ are both abelian, $c^{g}$ and $c^{\bar{g}}$ both commute with the non-trivial element $d$, hence
$[c^{g}, c^{\bar{g}}]=1$ by transitivity.  Conjugating by $\bar{g}$ gives $[c^{g\bar{g}^{-1}}, c]=1$.  From the CSA property it follows that
$[g\bar{g}^{-1}, c]=1$.  This equation may be rewritten as $c^{g}=c^{\bar{g}}$, which implies $a^{2}=\bar{a}^{2}$.  If $a=1$, then $\bar{a}=1$.
Otherwise, transitivity of commutation gives $(a\bar{a}^{-1})^{2}=1$ hence $a=\bar{a}$.  In either case, $(xa^{-a})^{\sigma_{1}\pi_{1}\gamma'}=1$.
For the commutators, transitivity of commutation gives $[\bar{g}g^{-1}, u]=1$ hence $[zg^{-1}, u]^{\sigma_{1}\pi_{1}\gamma'}=1$.

\emph{Form $c_{1}^{z_{1}} c_{2}^{z_{2}} d$ with $Q^{\pi_{1}'}$ abelian. } The argument discussed in the case $\chi(S_{1}')\leq -2$ applies.
\end{proof}

We note that if we rewrite $\Gamma_{T}$ as the coordinate group of a $\Gamma$-NTQ system $V=V_{1}\cup\ldots\cup V_{n}$, the solutions $\pi_{i}'$ give rise to
a solution sequence for $V$.

\subsection{Embedding fully residually $\Gamma$ groups}
We now combine results from the previous sections to obtain our main embedding theorem (Theorem~\ref{Thm:FinitelyManyEmbeddings}).
First, for any system of equations over $\Gamma$ we
construct a finite set of fully residually $\Gamma$ quotients through which all solutions must factor.



\comment{
\[
\xymatrix{
& F(Z,A) \ar[ld]_{\overline{\phantom{m}}} \ar[d]^{\rho_{i}} \\
G \ar[ddr]_{\psi} & F_{R(T_{i})} \ar[d]^{\phi} \\
& F(A) \ar[d]^{\pi} \\
& \Gamma
}
\]
}

\begin{lemma}\label{Lem:EmbeddingIntoCoordinateGroups}
There is an algorithm that, given a finitely presented group $G=\langle Z\gst S\rangle$,
produces
\begin{romanenumerate}
\item finitely many fully residually $\Gamma$ groups $G_{1},\ldots,G_{m}$, and 
\item homomorphisms $\alpha_{i}: G\rightarrow G_{i}$
\end{romanenumerate}
such that for every homomorphism $\psi: G\rightarrow \Gamma$ there exists a homomorphism
$\hat{\phi}: G_{i}\rightarrow \Gamma$ such that $\psi=\alpha_{i}\hat{\phi}$.  Further, each group $G_{i}$ has the form $G_{i}=\Gamma_{T_{i}}$ where
$T_{i}$ is an embeddable $\Gamma$-NTQ system and the homomorphism $\hat{\phi}$ is a $\Gamma$-homomorphism.
\end{lemma}
\begin{remark}
The lemma also holds in the category of $\Gamma$-groups, i.e. for systems of equations with constants over $\Gamma$.
\end{remark}
\begin{proof}
We will prove the lemma in the category of groups.  The proof for $\Gamma$-groups is similar.  Refer to Figure~\ref{Figure:CommDiagram} for a diagram of the maps constructed in the proof.
Note that for $G$ in the category of $\Gamma$-groups, the top group is $F(Z,A)$ rather than $F(Z)$. Denote by $\overline{\phantom{o}}: F(Z)\rightarrow G$
the canonical epimorphism.
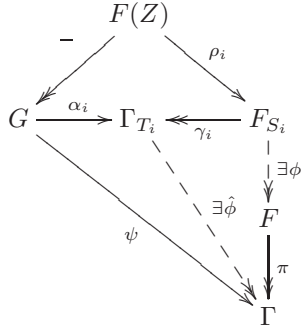
\begin{figure}[htbp]
\begin{center}
\[
\xymatrix{
&  F(Z) \ar@{->>}[ld]_{\overline{\phantom{\phi}}} \ar[rd]^{\rho_{i}} \\
G   \ar[ddrr]_{\psi} \ar[r]^{\alpha_{i}} & \Gamma_{T_{i}} \ar@{-->>}[rdd]^{\exists\hat{\phi}} &  F_{S_{i}} \ar@{->>}[l]^{\gamma_{i}} \ar@{-->>}[d]^{\exists\phi} &\\
& & F \ar@{->>}[d]^{\pi} \\
&   & \Gamma
}
\]
\caption{Homomorphisms for Lemma~\ref{Lem:EmbeddingIntoCoordinateGroups}.}
\label{Figure:CommDiagram}
\end{center}
\end{figure}

Construct the systems $S_{1}(X_{1},A),\ldots,S_{n}(X_{n},A)$
from Lemma~\ref{Lem:RipsSela1} along with the homomorphisms $\rho_{i}: F(Z)\rightarrow F_{S_{i}}$.  Applying Proposition~\ref{Prop:FreeNTQ},
we may assume that the systems $S_{1},\ldots,S_{n}$ are all $F$-NTQ systems and we have constructed the solution sequence $\Pi_{i}$ for each one.
For each system $S_{i}$, construct the corresponding $\Gamma$-NTQ system
$T_{i}$ from \S \ref{section:ConvertingNTQ} and the homomorphism $\gamma_{i}: F_{S_{i}}\rightarrow \Gamma_{T_{i}}$.
Define $\alpha_{i}: G\rightarrow \Gamma_{T_{i}}$ by  $\alpha_{i}=\overline{\rho_{i}\gamma_{i}}$,  as in Notation~\ref{Not:Lifting}. That is,
for any $\overline{u}\in G$,
\[
\overline{u}^{\alpha_{i}} = u^{\rho_{i}\gamma_{i}}.
\]

To check that $\alpha_{i}$ is well-defined, let
$u\in F(Z)$ with $\overline{u}=1$ (in $G$).  Since $u\in\nclofin{S}{F(Z)}$, there exist $s_{j}\in S$ and $w_{j}\in F(Z)$ such that
$u = \prod_{j=1}^{n} s_{j}^{w_{j}}$
hence
\[
u^{\rho_{i}\gamma_{i}} = \prod_{j=1}^{m} (s_{j}^{\rho_{i}\gamma_{i}})^{w_{j}^{\rho_{i}\gamma_{i}}}.
\]
We may assume that the relators of $G$ are in triangular form, hence $s_{j}=z_{1}z_{2}z_{3}$.
From the description of canonical representatives in Lemma~\ref{Lem:RipsSela1}, it follows that
\[
s_{j}^{\rho_{i}} = (x_{1}c_{1}x_{2}^{-1})(x_{2}c_{2}x_{3}^{-1})(x_{3}c_{3}x_{1}^{-1})
\]
where $c_{1} c_{2} c_{3}=1$ in $\Gamma$ and $x_{1},x_{2},x_{3}\in X_{i}$.
Hence
\[
s_{j}^{\rho_{i}}=(c_{1}c_{2}c_{3})^{x_{1}}.
\]
Since the relators of $\Gamma$ are elements of $T_{i}$ we have that $s_{j}^{\rho_{i}\gamma_{i}}=1$ in $\Gamma_{T_{i}}$
hence $u^{\rho_{i}\gamma_{i}}=1$ and $\alpha_{i}$ is well-defined.

Now let $\psi: G\rightarrow \Gamma$ be any homomorphism. For each
$i\in \{1,\ldots,n\}$, set
\[
\Phi_{i} = \{\overline{\rho_{i}\phi\pi}\sst \phi\in\Phi(\Pi_{i}) \},
\]
where $\overline{\rho_{i}\phi\pi}$ is defined as in Notation~\ref{Not:Lifting}.  As with $\alpha_{i}$, these homomorphisms are well-defined.
From Lemma~\ref{Lem:RipsSela1} and Proposition~\ref{Prop:FreeNTQ} we know that
\[
\mathrm{Hom} (G,\Gamma) = \bigcup_{i=1}^{n} \Phi_{i},
\]
hence $\psi=\overline{\rho_{i}\phi\pi}$ for some $\phi\in\Phi(\Pi_{i})$.  By Lemma~\ref{Lem:ConvertedNTQ}, $\phi\pi$ factors through $\Gamma_{T_{i}}$
hence there exists $\hat{\phi}: \Gamma_{T_{i}}\rightarrow \Gamma$ such that $\gamma_{i}\hat{\phi}=\phi\pi$.
Then for any $\overline{w}\in G$ we have
\[
\overline{w}^{\psi} = w^{\rho_{i}\phi\pi} =  w^{\rho_{i}\gamma_{i}\hat{\phi}} = \overline{w}^{\alpha_{i}\hat{\phi}}
\]
hence $\psi=\alpha_{i}\hat{\phi}$.
\end{proof}

\begin{corollary}\label{Cor:AtLeastOneInjective}
If $G$ is fully residually $\Gamma$, then there exists $i\in \{1,\ldots,m\}$ such that $\alpha_{i}$ is injective.  If $G$ is residually $\Gamma$, then for every $g\in G$ there
exists $i\in \{1,\ldots,m\}$ such that $g^{\alpha_{i}}\neq 1$.
\end{corollary}
\begin{proof}
If $G$ is fully residually $\Gamma$ but every $\alpha_{i}$ is non-injective, then there exists a set $\{g_{i}\}_{i=1}^{m}\subset G$  such that $g_{i}^{\alpha_{i}}=1$
for $i=1,\ldots,m$. This set cannot be discriminated since for every $\psi:G\rightarrow \Gamma$ there exists $i$ and $\phi$ such that $\psi=\alpha_{i}\hat{\phi}$,
hence $g_{i}^{\psi}=1$.
The second statement is clear.
\end{proof}

\begin{remark}
Though at least one of the homomorphisms $\alpha_{i}$ must be injective when $G$ is fully residually $\Gamma$, we are not aware of a method for determining which one
(there may be several).
\end{remark}

Before proceeding to the main result regarding embedding, we observe two corollaries.
\begin{corollary} \label{Cor:UnivTh}
The universal theory of a  torsion-free hyperbolic group is decidable.
\end{corollary}
\begin{proof} To show that the universal theory of $\Gamma$ is decidable it is equivalent to prove that the existential theory is decidable, so
we must give an algorithm to decide whether the conjunction
of a system of equations $U(X,A)=1$ and a system of inequations $V(X,A)\neq 1$ has a solution in $\Gamma.$
Apply the Lemma~\ref{Lem:EmbeddingIntoCoordinateGroups} to the $\Gamma$-group $G=\langle \Gamma, X\gst U\rangle$.
The conjunction has a solution if and only if there exists an index $i$ such that
the images of all elements from $V(X,A)$ are non-trivial in  $G_{i}.$ This we can check because the word problem in  $G_{i}$ is solvable.
\end{proof}

\begin{corollary}
Let $\Gamma=\GammaPresentation$ be a  torsion-free hyperbolic group.
There exists an algorithm that, given a  system of equations $U(X,A)=1$ over $\Gamma$, constructs a finite number of
$\Gamma$-NTQ systems $T_{i}(X_{i},A)=1$ over
$\Gamma$ that correspond to the fundamental sequences of solutions of $U(X,A)=1$
that satisfy the second restriction on fundamental sequences as in Section 7.9 in \cite{KM06}. Namely,
\begin{romanenumerate}
\item edge groups in the decompositions on each level are not mapped along this sequence into trivial elements,
\item images of QH subgroups on each level are non-cyclic,
\item images of rigid subgroups are non-cyclic.
\end{romanenumerate}
Each homomorphism  $\Gamma _{R(U)}\rightarrow \Gamma$ factors through one of these fundamental sequences.
 (Such fundamental sequences correspond to strict resolutions in Sela's terminology \cite{Sel09}.)
\end{corollary}

We may now prove the main result of the paper.
\begin{theorem}\label{Thm:FinitelyManyEmbeddings}
Let $\Gamma$ be any torsion-free hyperbolic group.
There is an algorithm that, given a finitely presented group $G$, constructs
\begin{romanenumerate}
\item finitely many groups
$H_{1},\ldots, H_{n}$, each given as a series of extensions of centralizers of $\Gamma$, and
\item homomorphisms $\phi_{i}: G\rightarrow H_{i}$,
\end{romanenumerate}
such that
\begin{arabicenumerate}
\item if $G$ is fully residually $\Gamma$, then at least one of the $\phi_{i}$ is injective, and
\item if $G$ is residually $\Gamma$, the map $\phi_{1}\times\ldots\times \phi_{n}: G\rightarrow H_{1}\times\ldots\times H_{n}$ is injective.
\end{arabicenumerate}
This also holds for $G$ in the category of $\Gamma$-groups.
\end{theorem}
\begin{proof}
Each of the groups $G_{i}$ constructed in Lemma~\ref{Lem:EmbeddingIntoCoordinateGroups} is
embeddable by Lemma~\ref{Lem:GammaTisNTQ} hence embeds effectively into a centralizer extension $H_{i}$ of $\Gamma$ by Corollary~\ref{Cor:NTQEmbedding}.
The result then follows from Corollary~\ref{Cor:AtLeastOneInjective}.
\end{proof}

As a corollary, we obtain a polynomial-time solution to the word problem in any finitely presented residually $\Gamma$ group.
\begin{corollary}\label{Cor:WPinGammaLimit}
Let $\Gamma$ be a torsion-free hyperbolic group and $G=\GPresentation$ any finitely presented group that is known to be residually $\Gamma$.
There is an algorithm that, given a word $w$ over the alphabet $Z^{\pm}$, decides whether or not $w=1$ in $G$ in time polynomial in $\wl{w}$.
\end{corollary}
\begin{proof}
We compute in advance the embedding $\phi: G \rightarrow H_{1}\times \ldots \times H_{n}$, i.e. we compute $z^{\phi}$ for each $z\in Z$.
Given the input word $w$, we need only
compute $w^{\phi}$ and solve the word problem in $H_{1}\times \ldots \times H_{n}$.  There is a fixed constant $L$ such that $\wl{\pi_{H_{i}}(w^{\phi})}\leq L \wl{w}$,
where $\pi_{H_{i}}$ is projection onto $H_{i}$,
so we have a polynomial reduction to $n$ word problems in the groups $H_{1}, \ldots, H_{n}$.  It then suffices to show that each $H_{i}$ has a polynomial
time word problem.

Let $H_{i}$ be formed by a sequence of $m$ extensions of centralizers and proceed by induction.  If $m=0$, then $H_{i}=\Gamma$ so the word problem in $H_{i}$ is decidable in
polynomial time.  Now assume that
\begin{equation}\label{Eqn:HNNinduction}
H_{i} = \langle H_{i}', t \gst [t, C(u)] \rangle
\end{equation}
where $u\in H_{i}'$ and $H_{i}'$ is formed from $\Gamma$ by a sequence of $m-1$ extensions of centralizers and has a polynomial time word problem.  Let $w$ be a word
in $H_{i}$.  It suffices to produce a reduced form for $w$ as an element of the HNN-extension (\ref{Eqn:HNNinduction}):
if any $t^{\pm 1}$ appears in the reduced form then $w\neq 1$, and if no $t^{\pm 1}$ appears then
$w\in H_{i}'$ and we check whether or not $w=1$ using the word problem algorithm for $H_{i}'$.

We produce a reduced form for $w$ by examining all subwords of the form $t v t^{-1}$ and $t^{-1} v t$ where no $t^{\pm 1}$ appears in $v$, and
making reductions
\[
t v t^{-1} \rightarrow v, \;\;\;\; t^{-1} v t \rightarrow v
\]
whenever $v\in C_{H_{i}'}(u)$.  The element $v$ is in $C_{H_{i}'}(u)$ if and only if $[v,u]=1$ in $H_{i}'$,
which is an instance of the word problem in $H_{i}'$ and so may be checked in polynomial time.
It is clear that we need only examine a polynomial number of subwords $t v t^{-1}$ and
$t^{-1} v t$ before reaching a reduced form.
\end{proof}

\bibliographystyle{plain}
\bibliography{embedding}

\begin{thebibliography}{10}

\bibitem{BMR99}
G.~Baumslag, A.~Myasnikov, and V.~Remeslennikov.
\newblock Algebraic geometry over groups. {I}. {A}lgebraic sets and ideal
  theory.
\newblock {\em J. Algebra}, 219(1):16--79, 1999.

\bibitem{Bum04}
I.~Bumagin.
\newblock The conjugacy problem for relatively hyperbolic groups.
\newblock {\em Algebr. Geom. Topol.}, 4:1013--1040, 2004.

\bibitem{Dah03}
F.~Dahmani.
\newblock Combination of convergence groups.
\newblock {\em Geom. Topol.}, 7:933--963 (electronic), 2003.

\bibitem{Dah08}
F.~Dahmani.
\newblock Finding relative hyperbolic structures.
\newblock {\em Bull. Lond. Math. Soc.}, 40(3):395--404, 2008.

\bibitem{Dah09}
F.~Dahmani.
\newblock Existential questions in (relatively) hyperbolic groups.
\newblock {\em Israel J. Math.}, 173:91--124, 2009.

\bibitem{EH01}
D.~Epstein and D.~Holt.
\newblock Computation in word-hyperbolic groups.
\newblock {\em Internat. J. Algebra Comput.}, 11(4):467--487, 2001.

\bibitem{Far98}
B.~Farb.
\newblock Relatively hyperbolic groups.
\newblock {\em Geom. Funct. Anal.}, 8(5):810--840, 1998.

\bibitem{Gro05}
D.~Groves.
\newblock Limit groups for relatively hyperbolic groups. {II}.
  {M}akanin-{R}azborov diagrams.
\newblock {\em Geom. Topol.}, 9:2319--2358, 2005.

\bibitem{Gro09}
D.~Groves.
\newblock Limit groups for relatively hyperbolic groups. {I}. {T}he basic
  tools.
\newblock {\em Algebr. Geom. Topol.}, 9(3):1423--1466, 2009.

\bibitem{KM98a}
O.~Kharlampovich and A.~Myasnikov.
\newblock Irreducible affine varieties over a free group. {I}. {I}rreducibility
  of quadratic equations and {N}ullstellensatz.
\newblock {\em J. Algebra}, 200(2):472--516, 1998.

\bibitem{KM98b}
O.~Kharlampovich and A.~Myasnikov.
\newblock Irreducible affine varieties over a free group. {II}. {S}ystems in
  triangular quasi-quadratic form and description of residually free groups.
\newblock {\em J. Algebra}, 200(2):517--570, 1998.

\bibitem{KM05JSJ}
O.~Kharlampovich and A.~Myasnikov.
\newblock Effective {JSJ} decompositions.
\newblock In {\em Groups, languages, algorithms}, volume 378 of {\em Contemp.
  Math.}, pages 87--212. Amer. Math. Soc., Providence, RI, 2005.

\bibitem{KM05Implicit}
O.~Kharlampovich and A.~Myasnikov.
\newblock Implicit function theorem over free groups.
\newblock {\em J. Algebra}, 290(1):1--203, 2005.

\bibitem{KM06}
O.~Kharlampovich and A.~Myasnikov.
\newblock Elementary theory of free non-abelian groups.
\newblock {\em J. Algebra}, 302(2):451--552, 2006.

\bibitem{KM09}
Olga Kharlampovich and Alexei Myasnikov.
\newblock Limits of relatively hyperbolic groups and {L}yndon's completions.
\newblock {\em J. Eur. Math. Soc. (JEMS)}, 14(3):659--680, 2012.

\bibitem{LS77}
R.~Lyndon and P.~Schupp.
\newblock {\em Combinatorial Group Theory}.
\newblock Springer, 1977.

\bibitem{MR96}
A.~Myasnikov and V.~Remeslennikov.
\newblock Exponential groups. {II}. {E}xtensions of centralizers and tensor
  completion of {CSA}-groups.
\newblock {\em Internat. J. Algebra Comput.}, 6(6):687--711, 1996.

\bibitem{Osi06IJAC}
D.~Osin.
\newblock Elementary subgroups of relatively hyperbolic groups and bounded
  generation.
\newblock {\em Internat. J. Algebra Comput.}, 16(1):99--118, 2006.

\bibitem{Osi06Memoirs}
D.~Osin.
\newblock Relatively hyperbolic groups: intrinsic geometry, algebraic
  properties, and algorithmic problems.
\newblock {\em Mem. Amer. Math. Soc.}, 179(843):vi+100, 2006.

\bibitem{RS95}
E.~Rips and Z.~Sela.
\newblock Canonical representatives and equations in hyperbolic groups.
\newblock {\em Invent. Math.}, 120(3):489--512, 1995.

\bibitem{Sel01}
Z.~Sela.
\newblock Diophantine geometry over groups. {I}. {M}akanin-{R}azborov diagrams.
\newblock {\em Publ. Math. Inst. Hautes \'Etudes Sci.}, (93):31--105, 2001.

\bibitem{Sel06}
Z.~Sela.
\newblock Diophantine geometry over groups. {VI}. {T}he elementary theory of a
  free group.
\newblock {\em Geom. Funct. Anal.}, 16(3):707--730, 2006.

\bibitem{Sel09}
Z.~Sela.
\newblock Diophantine geometry over groups. {VII}. {T}he elementary theory of a
  hyperbolic group.
\newblock {\em Proc. Lond. Math. Soc. (3)}, 99(1):217--273, 2009.

\end{thebibliography}

\end{document}